\documentclass[final,reqno,twoside,onecolumn]{siamltex}
\usepackage{amscd}
\usepackage{amssymb}
\usepackage{amsmath,xcolor}
\usepackage{mathrsfs}
\usepackage{graphicx}
\usepackage{calc}%
\usepackage{multirow}
\usepackage{placeins}
\usepackage[T1]{fontenc} 
\usepackage{fourier} 
\usepackage{cite}

\newcommand{\Bb}{\mathcal B}

\newcommand{\Dt}{\Delta t}

\newcommand{\varep}{\varepsilon}
\newcommand{\gabs}[1]{\left | #1 \right|}
\newcommand{\norm} [1]{\| {#1}\|}

\newcommand{\R} {\mathbb R}

\newcommand{\M} {\mathcal M}
\newcommand{\Ro} {\mathcal R}
\newcommand{\intb}[1]{\left\langle #1 \right\rangle}
\newcommand{\intd}[1]{\left( #1 \right)}

\newcommand{\ddt}{\frac d{dt} }
\newcommand{\eqdef}{\overset{\mathrm{def}}{=\joinrel=}}
\def\beq{\begin{equation}}                                                                                                                                                                                                                                                                                                                                                                                                                                                                                                                                                                         
\def\eeq{\end{equation}} 
\def\beqs{\begin{equation*}}
\def\eeqs{\end{equation*}}
\def\bals{\begin{align*}}
\def\eals{\end{align*}}
\def\bspl{\begin{split}}
\def\espl{\end{split}}
\def\myclearpage{}
\title{A mixed finite element approximation for Darcy-Forchheimer flows of slightly~compressible~fluids}
\author{Thinh Kieu \footnotemark[2]  }
\date{today}

\begin{document}
\maketitle
\renewcommand{\thefootnote}{\fnsymbol{footnote}}
\footnotetext[2]{Department of Mathematics, University of North Georgia, Gainesville Campus, 3820 Mundy Mill Rd., Oakwood, GA 30566, U.S.A. ({\tt thinh.kieu@ung.edu}).}               
               
\begin{abstract} In this paper, we consider the generalized Forchheimer flows for slightly compressible fluids in porous media. 
Using Muskat's and Ward's general form of Forchheimer equations, we describe the flow of a single-phase fluid in $\R^d, d\ge 2$ by a nonlinear degenerate system of density and momentum. A mixed finite element method is proposed for the approximation of the solution of the above system.   
The stability of the approximations are proved; the error estimates are derived for the numerical approximations for both continuous and discrete time procedures.  
 The continuous dependence of numerical solutions on physical parameters are demonstrated. 
Experimental studies  are presented regarding convergence rates and showing the dependence of the solution on the physical parameters.  
\end{abstract}



\pagestyle{myheadings}
\thispagestyle{plain}
\markboth{Thinh Kieu}{Mixed FEM for Darcy-Forchheimer flows in porous media}
           
\myclearpage    
\section {Introduction }

The fluid flow through porous materials, e.g. soil, sand, aquifers, oil reservoir, plants, wood, bones, etc.,  is a great interest in the research community such as engineering, oil recovery, environmental and groundwater hydrology and medicine. Darcy law, which is the linear relation between the velocity vector and the pressure gradient, is used to describe fluid flow under low velocity and low porosity conditions, see in  \cite{aziz1979petroleum}. It has been observed from many experiments that when the fluid's velocity is high and porosity is nonuniform, the Darcy's law becomes inadequate. Consequently, the attention has been attracted to the nonlinear equations for describing of this kind of flow. Dupuit and Forchheimer proposed a modified equation, known as Darcy- Forchheimer equation or generalized Forchheimer equation, by adding the nonlinear terms of velocity to Darcy law, see~\cite{ForchheimerBook}. Since then, there have been a growing number of articles studying  Darcy- Forchheimer equation in theoretical  studies (e.g.\cite{MR2433781, MR2761004,MR2459923,MR2232258,MR2376202,MR2298682}) and numerical studies (e.g.\cite{MR2592414,DW93, EJP05, GW08, KP99, salas2013analysis}).

It is well known that the mixed finite element method is among the popular numerical methods for the modeling flow in porous media because it produces the accurate results for both scalar (density or pressure) and vector (velocity or momentum) functions, see~\cite{EJP05}. An analysis of mixed finite element method to a Darcy-Forchheimer steady state model was well studied in \cite{HRH12, salas2013analysis}. The mixed methods for a nondegenerate system modeling flows in porous media was studied in \cite {DW93, EJP05, GW08, KP99}. The authors in \cite {ATWZ96, WCD00,FS95,FS04} analyzed the mixed finite element approximations of the nonlinear degenerate system modeling water-gas flow in porous media. In their analysis, the Kirchhoff transformation is used to move the nonlinearity from coefficients to the gradient.   

The objective of this paper is to analyze mixed finite element approximations to the solutions of the system of equations modeling the flows of a single-phase compressible fluid in porous media subject to the generalized Forchheimer law. This is a nonlinear degenerate system with coefficients depending on the density gradient and degenerating  to zero as it approaches to infinity. The Kirchhoff transformation is not applicable for this system. For our degenerate equations, we combine the techniques developed in our previous works in \cite {HI1,HI2,HIKS1, HK1,HK2,HKP1, IK1, K1} and utilize the special structures of the equations to obtain the stability of the approximated solution and the continuous dependence of the solution on parameters. The error estimates also derives for the density and momentum.      

  The paper is structured as follows.  
  In section \S \ref{Intrsec}, we introduce the notations and the relevant results.
 In section \S \ref{MFEM}, we defined a numerical approximation using mixed finite element approximations and the implicit backward difference time discretization to the initial boundary value problem (IVBP)~\ref{rho:eq}. 
 In section \S \ref{Bsec}, we establish many estimates of the energy type norms for the approximate solution $(\rho_h, m_h)$ to the IVBP problem~\eqref{semidiscreteform} in Lebesgue norms in terms of the boundary data and the initial data.
In section \S \ref{dependence-sec}, we focus on proving the continuous dependence of the solution on the coefficients of Forchheimer polynomial $g$. In order to obtain this, we first establish the perturbed monotonicity for our degenerate partial differential equations, see  in \eqref{quasimonotone}. It is then proved in Theorem~\ref{DepCoeff} that the difference between the two solutions, which corresponds to two different coefficient vectors $\vec a_1$ and  $\vec a_2$ is estimated in terms of $|\vec a_1-\vec a_2|$, see in \eqref{ssc2}.
In section \S \ref{err-sec}, we study in Theorem~\ref{err-thm} the convergence and in Theorem~\ref{par-thm} the dependence on coefficients of Forchheimer polynomial of the approximated solution to the problem~\eqref{semidiscreteform}. Furthermore, we can specify the convergent rate.          
In section \S \ref{err-ful-sec}, we study the fully discrete version of problem~\eqref{semidiscreteform}. In Lemma~\ref{stab-appr}, the stability of the approximated solution is proved. Theorems~\ref{Err-ful}~and~\ref{Dep-ful} are for studying the error estimates and the continuous dependence on parameters of the numerical solution.   
In section \S \ref{Num-result}, the numerical experiments in the two-dimensions using the standard finite elements $\mathbb P_1$ are presented regarding the convergence rates and the dependence of the solution on the physical parameters.

\section {Preliminaries and auxiliaries} \label{Intrsec}  We consider a fluid in porous medium  occupying a bounded domain $\Omega\subset \R^d,$ $d\ge 2$ with boundary $\Gamma$. Let $x\in\R^d $, $0<T<\infty$ and $t\in (0,T]$ be the spatial and time variables respectively. The fluid flow has velocity $v(x,t)\in \R^d$, pressure $p(x,t)\in\R$ and density $\rho(x,t)\in \R_+$.   

The  Darcy--Forchheimer equation is studied in \cite{ABHI1,HI1,HI2} of the form 
\beq\label{gF}
-\nabla p =\sum_{i=0}^N a_i |v|^{\alpha_i}v. 
\eeq  
where $N\ge 1$, $\alpha_0=0<\alpha_1<\ldots<\alpha_N$ are real (not necessarily integral) numbers, and the coefficients satisfy  $a_0,~a_N~>~0$ and $a_1,\ldots,a_{N-1}\ge 0$.

In order to take into account the presence of density in the generalized Forchheimer equation, we modify \eqref{gF} using the dimensional analysis by Muskat \cite{Muskatbook} and Ward \cite{Ward64}. They proposed the following equation for both laminar and turbulent flows in porous media:
\beq\label{W}
-\nabla p =F(v^\alpha \kappa^{\frac {\alpha-3} 2} \rho^{\alpha-1} \mu^{2-\alpha}),\text{ where  $F$ is a function of one variable.}
\eeq 
In particular, when $\alpha=1,2$, Ward \cite{Ward64} established from experimental data that
\beq\label{FW} 
-\nabla p=\frac{\mu}{\kappa} v+c_F\frac{\rho}{\sqrt \kappa}|v|v,\quad \text{where }c_F>0.
\eeq

Combining  \eqref{gF} with the suggestive form \eqref{W} for the dependence on $\rho$ and $v$, we propose the following equation 
 \beq\label{FM}
-\nabla p= \sum_{i=0}^N a_i \rho^{\alpha_i} |v|^{\alpha_i} v.
 \eeq
Here, the viscosity and permeability are considered constant, and we do not specify the dependence of $a_i$'s on them.

Multiplying equation \eqref{FM} to $\rho$, we obtain 
 \beq\label{eq1}
 g( |\rho v|) \rho v   =-\rho\nabla p,
 \eeq
where the function $g$ is a generalized polynomial with non-negative coefficients. More precisely,  the function $g:\R_+\rightarrow\R_+$ is of the form
\beq\label{eq2}
g(s)=a_0s^{\alpha_0} + a_1s^{\alpha_1}+\cdots +a_Ns^{\alpha_N},\quad s\ge 0, 
\eeq 
where $N\ge 1$, $\alpha_0=0<\alpha_1<\ldots<\alpha_N$ are real (not necessarily integral) numbers. The coefficients satisfy  $a_0,~a_N~>~0$ and $a_1,\ldots,a_{N-1}\ge 0$.
The number $\alpha_N$ is the degree of $g$ and is denoted by $\deg(g)$.
Denote the vectors of powers and coefficients by $\vec \alpha=(\alpha_0,\ldots,\alpha_N)$ and
 $\vec a=(a_0,\ldots,a_N)$. 
The class of functions $g(s)$ as in \eqref{eq2} is denoted by FP($N,\vec \alpha$), which is the abbreviation of ``Forchheimer polynomials.'' When the function $g$ in \eqref{eq2} belongs to FP($N,\vec \alpha$), it is referred to as the Forchheimer polynomial. 

For slightly compressible fluids, the state equation is
\beq
\frac{d\rho}{dp}=\frac {\rho}{\kappa}, \quad \kappa=\text{const}., \kappa\gg 1
\eeq
which yields 
\beq \label{eq3}
\rho \nabla p=\kappa\nabla \rho.
\eeq
It follows form \eqref{eq2} and \eqref{eq3} that  
 \beq\label{Gs}
 g( |\rho v|) \rho v   =-\kappa\nabla \rho.
 \eeq
Solving  for $\rho v$ from \eqref{Gs} gives 
\beq\label{ru} 
\rho v=- \kappa K(|\kappa\nabla \rho |)\nabla \rho,
\eeq
where the function $K: \R_+\rightarrow\R_+$ is defined for $\xi\ge 0$ by
\beq\label{Kdef}
K(\xi)=\frac{1}{g(s(\xi))}, \text{ with }  s=s(\xi) \text{  being the unique non-negative solution of } sg(s)=\xi.
\eeq  

The continuity equation is
\beq\label{con-law}
\phi\rho_t+{\rm div }(\rho v)=f,
\eeq
where $\phi$ is the porosity, and $f$ is the external mass flow rate . 

By rescaling the coefficients of the conductivity function $K(\cdot)$, we can assume $\kappa=1$. We introduce the momentum variables 
$
m=\rho v. 
$

Combining \eqref{ru} and \eqref{con-law}, we obtain equations in a density-momentum formulation 
\beq\label{utporo}
\begin{cases}
 m   +K( |\nabla \rho|)\nabla \rho =0, \\
 \phi\rho_t+ \nabla\cdot  m=f.
\end{cases}
\eeq

We will study the initial boundary value problem (IVBP) associated with the coupled system \eqref{utporo}. 
We will derive estimates for the solution, and establish their continuous dependence on the coefficients of the function $g(s)$ in \eqref{eq2}. As seen below, the generalized permeability tensor $K$ is degenerate to zero when $|\nabla \rho|\to \infty$.  This was not considered in existing papers, e.g.  in \cite{JD93, KP99, HRH12}. Therefore, it creates an additional challenge  and requires extra care in the proof and analysis.

 Let $g=g(s,\vec a)$ in FP($N,\vec \alpha$). 
The following numbers  are frequently used in our calculations:
\begin{align}
\label{Ag} 
&\chi(\vec a)=\max\Big \{a_0,a_1,\ldots,a_N,\frac1{a_0},\frac1{a_N} \Big \}\in [1,\infty),\\
 \label{adef}
 &  a=\frac{\alpha_N}{\alpha_N+1}= \frac{\deg \ g}{\deg \ g+1}\in (0,1),
  \quad \beta =2-a\in (1,2),\quad  \lambda= \frac {2-a}{1-a}=\frac{\beta}{\beta-1}\in(2,\infty).
  \end{align}

\begin{lemma}[cf. \cite{ABHI1, HI1}, Lemma 2.1] Let $g(s,\vec a)$ be in class FP($N,\vec \alpha$). For any $\xi\ge 0$, One has 
 \begin{enumerate}
\item [(i)] $K: [0,\infty)\to (0,a_0^{-1}]$ and it decreases in $\xi$.  

\item[(ii)] For any $n\ge 1$, the function $K(\xi)\xi^n$ increases and $K(\xi)\xi^n\ge 0$.

\item[(iii)]  Type of degeneracy  \beq\label{i:ineq1}  
c_0(1+\xi)^{-a}\le K(\xi, \vec a)\le c_1 (1+\xi)^{-a}.
\eeq

\item[(iv)]  For all $n\ge 1,\delta>0,$ 
\beq\label{i:ineq2} 
c_2 (\xi^{m-a}-\delta^{m-a}) \le  K(\xi,\vec a)\xi^m\le c_3\xi^{m-a},
\eeq
where $c_0, c_1, c_2, c_3$ depend on $N$ and $\chi(\vec a), \alpha_N$ only.

In particular, when $m=2$, $\delta=1$, one has
\beq\label{iineq2} 
c_2(\xi^{2-a}-1)\le K(\xi,\vec a)\xi^2\le c_3\xi^{2-a}.
\eeq

 \item[(v)]  Relation with its derivative \beq\label{i:ineq3} -aK(\xi)\leq K'(\xi)\xi\leq 0. \eeq
\end{enumerate}
\end{lemma}

We define
\beq\label{Hxi} H(\xi,\vec a)=\int_0^{\xi^2}K(\sqrt s,\vec a)ds \quad \hbox{for } \xi\ge 0.\eeq

When vector $\vec a$ is fixed, we denote $K(\cdot,\vec a)$ and $H(\cdot,\vec a)$ by $K(\cdot)$ and $H(\cdot)$, respectively.
The function $H(\xi)$ can be compared with $\xi$ and $K(\xi)$ by 
\beq\label{Hcompare} 
K(\xi)\xi^2\leq H(\xi)\le 2K(\xi)\xi^2,\quad c_4(\xi^{2-a}-1)\leq H(\xi)\le  c_5\xi^{2-a},
\eeq
where $c_4,c_5>0$ depend on $\chi(\vec a)$.

For convenience, we use the following notations: let $\vec{x}=(x_1, x_2, \ldots, x_d)$ and $\vec{x}'=(x'_1, x'_2, \ldots, x'_d)$ be two arbitrary vectors of the same length, including possible length $1$. We denote by $\vec{x}\vee \vec{x}'$ and $\vec{x}\wedge \vec{x}'$ their maximum and minimum vectors, respectively, with components
$(\vec{x}\vee \vec{x}')_j=\max\{x_j, x'_j\}$  and $(\vec{x}\wedge \vec{x}')_j=\min\{ x_j, x'_j\}$.

\begin{lemma} Let $S=\{ \vec{a} =( a_0, \ldots, a_N): a_0, a_N>0, a_1, \ldots, a_{N-1}\ge 0\}$ be the set of admissible $\vec a$. For any coefficient vectors  $\vec a$, $\vec a'\in S$,  and  any  $y,y'\in \mathbb R^d$, one has 

(i)
\beq\label{Umono}
\begin{split}
|K(|y|,\vec a) y-K(|y'|,\vec a')y'|&\le (1+a) |y-y'| \int_0^1  K(|\gamma(t)|,\vec{b}(t)) dt\\
 &+ d_0 (|y|\vee |y'|)  |\vec a -\vec a'| \int_0^1 K(|\gamma(t)|,\vec{b}(t)) dt. 
\end{split}
\eeq

(ii)
\beq
\begin{aligned}\label{quasimonotone}
(K(|y|,\vec a) y-K(|y'|,\vec a')y')\cdot (y-y')
&\ge (1-a) |y-y'|^2 \int_0^1  K(|\gamma(t)|,\vec{b}(t)) dt  \\
&-d_0 (|y|\vee |y'|)   |\vec a-\vec a'| |y-y'| \int_0^1 K(|\gamma(t)|,\vec{b}(t)) dt  , 
\end{aligned}
\eeq
where $a\in(0,1)$ is defined in \eqref{adef},  
$$\gamma(t)=ty+(1-t)y', \quad \vec{b}(t)=\vec a+(1-t)\vec a',\quad
d_0=N\left(\min\{a_0, a'_0,(1+\alpha_N)a_N,(1+\alpha_N)a'_N\}\right)^{-1}.
$$

In particular, if $\vec a = \vec a'$ then \eqref{Umono} and \eqref{quasimonotone} become

(iii)
\beq\label{Lipchitz}
   \left|K(|y'|)y' -K(|y|)y\right| \leq (1+a) |y' -y| \int_0^1 K(|\gamma(t)|, \vec{a}) dt.
\eeq  
 
(iv)
\beq\label{monotone0}
   (K(|y'|)y' -K(|y|)y)\cdot(y'-y) \geq (1-a)|y' -y|^2 \int_0^1 K(|\gamma(t)|, \vec{a}) dt.
\eeq 
\end{lemma}
\begin{proof} 
(i). Let  $\vec a$, $\vec a'\in S$ and  $y,y' , \vec {k} \in \mathbb R^d$.

{\it Case 1}: The origin does not belong to the segment connecting  $y'$ and  $y$. Define 
 $$z(t)=K(|\gamma(t)|,\vec{b}(t))\, \gamma(t)\cdot \vec {k} .$$
 By the Mean Value Theorem,  we have 
\beqs
I
\eqdef [K(|y|,\vec a)y-K(|y'|,\vec a')y']\cdot \vec {k}
=z(1)-z(0)=\int_0^1 z'(t) dt.
\eeqs
Elementary calculations give
\beq\label{Ifm}
I=\int_0^1 f_1(t)dt + \int_0^1 f_2(t) dt\eqdef I_1+I_2,
\eeq
where
\begin{align*}
f_1(t)&=K(|\gamma(t)|,\vec{b}(t))(y-y')\cdot \vec{k}
 +K_\xi(|\gamma(t)|,\vec{b}(t))\frac{\gamma(t)\cdot(y-y')}{|\gamma(t)|} \gamma(t)\cdot \vec{k},\\
f_2(t)&= K_{\vec a}(|\gamma(t)|,\vec{b}(t))(\vec a-\vec a')\gamma(t)\cdot \vec{k} .
\end{align*}
 $\bullet$ \emph{Estimation of $I_1$.}  Since 
\beqs
\begin{split}
|f_1(t)| \le \left |K(|\gamma(t)|,\vec{b}(t))(y-y')
 +K_\xi(|\gamma(t)|,\vec{b}(t))\frac{\gamma(t)\cdot(y-y')}{|\gamma(t)|} \gamma(t)  \right|\,  |\vec{k}| \le (1+a)K(|\gamma(t)|,\vec{b}(t)) |y-y'|\, |\vec{k}|, 
\end{split}
\eeqs
we find that
\beq\label{UI1}
I_1 \le \int_0^1 |f_1(t)| dt \le  (1+a) |y-y'|\, |\vec{k}|\, \int_0^1 K(|\gamma(t)|,\vec{b}(t)) dt.
\eeq
$\bullet$ \emph{Estimation of $I_2$.} We find the partial derivative of $K(\xi,\vec{a})$ in $\vec{a}$.
For $i=0,1,\ldots,N$,  taking the partial derivative in $a_i$  of the identity $K(\xi,\vec{a})=1/g(s(\xi,\vec{a}),\vec{a})$, we find that
\begin{align*}
 K_{a_i}(\xi,\vec{a})=-\frac{g_{a_i}+g_s\cdot  s_{a_i}}{g^2}=-K(\xi,\vec{a})\frac{g_{a_i}+g_s\cdot  s_{a_i}}{g}.
\end{align*}
From  $sg(s,\vec a)=\xi $, we have for $i=0,1,\dots, N$,
$
s_{a_i}\cdot g+s\cdot ( g_{a_i}+g_s\cdot  s_{a_i} )=0,
$ which implies $
s_{a_i}=\frac{-s\cdot g_{a_i}}{g+s\cdot g_s}.
$

Hence,  we obtain
\beq\label{Kai}
K_{a_i}(\xi,\vec{a})=-K(\xi,\vec{a})\frac{g_{a_i}+g_s\cdot  \frac{-s\cdot g_{a_i}}{g+s\cdot g_s}}{g}
=-K(\xi,\vec{a})\frac{ g_{a_i}}{g+s\cdot g_s}=-K(\xi,\vec{a})\frac{ s^{\alpha_i}}{g+s\cdot g_s}.
\eeq
This shows that
\beqs
 \sum_{i=0}^N |K_{a_i}(\xi,\vec{a})|
\le K(\xi,\vec{a}) \frac{ 1+s^{\alpha_1}+\cdots+s^{\alpha_N} }{ a_0+(1+\alpha_1)a_1s^{\alpha_1}+\cdots+(1+\alpha_N)a_Ns^{\alpha_N} }.
\eeqs

Using inequality $x^\gamma\le x^{\gamma_1}+x^{\gamma_2}\quad\text{for all }x>0,\ \gamma_1\le \gamma \le \gamma_2$, we have
$
s^{\alpha_1},\ldots, s^{\alpha_{N-1}}\le 1+s^{\alpha_N}.
$
Hence,
\beqs
 \sum_{i=0}^N |K_{a_i}(\xi,\vec{a})|
\le K(\xi,\vec{a}) \frac{N(1+s^{\alpha_N}) }{a_0+(1+\alpha_N)a_N s^{\alpha_N} }
\le d (\vec{a})K(\xi,\vec{a}),
\eeqs
where
$
d(\vec{a})=N\left(\min\{a_0,(1+\alpha_N)a_N\}\right)^{-1}.
$
Thus,
\beq\label{Ka}
|K_{\vec a}(\xi,\vec{a})| \le d(\vec{a}) K(\xi,\vec{a}).
\eeq

Using the estimate \eqref{Ka}, we bound 
\beq\label{ih2}
\begin{aligned}
| f_2(t)|
\le |K_{\vec{a}}(|\gamma(t)|,\vec{b}(t))|\ |\vec a -\vec a'|\  | \gamma(t) |\ |\vec{k}|\le d(\vec{b}(t)) K(|\gamma(t)|,\vec{b}(t))\ | \gamma(t) |\  |\vec a -\vec a'|  \ |\vec{k}|.
\end{aligned}
\eeq

Since $a_i, a'_i$ is positive for $i=0,N$, the number $d(\vec{b}(t))$, for all $t\in[0,1]$, can be bounded  by $d(\vec{b}(t)) \le  d_0.$ Using the fact $|\gamma(t)|\le |y|\vee |y'|$,  \eqref{ih2} yields
\beq\label{Uh2}
| f_2(t)|
\le d_0 K(|\gamma(t)|,\vec{b}(t))  (|y|\vee |y'|)   |\vec a -\vec a' | \ |\vec{k}|,
\eeq
and consequently,
\beq \label{UI2}
I_2\le  \int_0^1 | f_2(t)|dt 
\le  d_0 (|y|\vee |y'|)  |\vec a -\vec a'| |\vec{k}|\int_0^1 K(|\gamma(t)|,\vec{b}(t))| dt. 
\eeq

Thus, we obtain \eqref{Umono} by combining \eqref{Ifm}, \eqref{UI1} and \eqref{UI2}.

{\it Case 2}: The origin belongs to the segment connect $y', y$. We replace $y'$ by $y'+y_\varep$ so that $0\notin [y'+y_\varep, y]$ and $y_\varep \to 0$ as $\varep \to 0$. Apply the above inequality for $y$ and $y'+y_\varep$, then let $\varep\to 0$.

{(ii)} If $\vec{k} = y-y'$ then 
\begin{align*}
K_\xi(|\gamma(t)|,\vec{b}(t))\ge -a \frac{K(|\gamma(t)|,\vec{b}(t))}{|\gamma(t)|}.
\end{align*}
This implies 
\beq\label{ks}
f_1(t) \ge (1-a) K(|\gamma(t)|,\vec{b}(t))|y-y'|^2.
\eeq
It follows \eqref{ks} that 
\beq\label{int-h1}
\int_0^1 f_1(t) dt
  \ge  (1-a)|y'-y|^2 \int_0^1 K(|\gamma(t)|,\vec{b}(t)) dt,
\eeq
 and from \eqref{Uh2}, we see that
\beq \label{I2geq}
I_2\ge - \int_0^1 | f_2(t)|dt 
\ge - d_0 (|y|\vee |y'|)   |\vec a -\vec a'|\, |y'-y|\int_0^1 K(|\gamma(t)|,\vec{b}(t))| dt  .
\eeq
Thus, we obtain \eqref{quasimonotone} by combining \eqref{Ifm}, \eqref{int-h1} and \eqref{I2geq}.
\end{proof}

Now we derive the trace estimates suitable for our nonlinear problem.   

\begin{lemma}\label{traceest} Assume $v(x)$ is a function defined on $\Omega$.

 (i) If $|v|\in W^{1,1}(\Omega)$ then there is a positive constant $C_1$ depending on $\Omega, \beta$ such that for all $\varep>0$,      
 \beq\label{sec2}
\int_{\Gamma} |v|d\sigma\le C_1\norm{v} + \varep\norm{\nabla v}_{0,\beta}^\beta +C_1\varep^{-\frac 1{\beta-1}}.
\eeq

(ii) If $u\in L^\infty(\Gamma)$ and $|v|\in W^{1,1}(\Omega)$ then  there exists  $C_2(\Omega,\beta)>0$ such that for all $\varep>0,$
\beq\label{bnd-eps}
\gabs{ \langle u, v\rangle } \le \varep \left(\norm{v}^2+ \norm{\nabla v}_{0,\beta}^{\beta}\right)+C_2\Big( \varep^{-1}\norm{u}_{L^\infty(\Gamma)}^2 +\varep^{-\frac 1{\beta-1}} \norm{u}_{L^\infty(\Gamma)}^{^\lambda}\Big).
\eeq

Consequently, 
\beq\label{bnd-est}
\gabs{ \langle u, v\rangle } \le \frac 14 \left(\norm{v}^2+ \norm{\nabla v}_{0,\beta}^{\beta}\right)+C_3\Big( 1 + \norm{u}_{L^\infty(\Gamma)}^{^\lambda}\Big)
\eeq
for a constant $C_3>0$.
\end{lemma}
\begin{proof} We recall the trace theorem 
\beqs
\int_{\Gamma} |v| dx \le C_*\int_\Omega |v| dx +C_*\int_\Omega |\nabla v|dx, 
\eeqs
for all $v\in W^{1,1}(\Omega),$ where $C_*$ is a positive constant depending on $\Omega$.  Using Young's inequality, we obtain  \eqref{sec2}.
 
(ii) We have 
\beq
\gabs{ \langle u, v\rangle }
\le \norm{u}_{L^\infty(\Gamma)}\int_{\Gamma} |v|d\sigma.
\eeq
Using \eqref{sec2}  and Young's inequality give  
\beqs
\gabs{ \langle u, v\rangle } \le \norm{u}_{L^\infty(\Gamma)}\Big( \delta\norm{v}^2+ \frac {C_*^2|\Omega|}{4} \delta^{-1} + \delta\norm{\nabla v}_{0,\beta}^{\beta} +C_1\delta^{-\frac 1{\beta-1}}\Big).
\eeqs

If $\norm{u}_{L^\infty(\Gamma)} =0$  then \eqref{bnd-eps} clearly holds true.

Otherwise, selecting $\delta = \varep  \norm{u}_{L^\infty(\Gamma)}^{-1} $ and the fact that $\frac 1{\beta-1} + 1=\lambda$,  we obtain \eqref{bnd-eps}.

Estimate \eqref{bnd-est} follows by choosing $\varep = 1/4$ in \eqref{bnd-eps} and using Young's inequality.  
\end{proof}

We recall a discrete version of Gronwall Lemma in backward difference form, which is useful later. It can be proven without much difficulty by following the ideas of the proof in Gronwall Lemma.
\begin{lemma}\label{DGronwall}
Assume $\ell>0, 1-\ell\Dt>0$ and the nonnegative sequences $\{a_n\}_{n=0}^\infty$, $\{g_n\}_{n=0}^\infty$  satisfying 
\beqs
\frac{a_n-a_{n-1}}{\Dt} - \ell a_n \le g_n, \quad n=1,2,3\ldots
\eeqs
then 
\beq\label{Gineq}
a_n\le (1-\ell \Dt)^{-n} \left(  a_0 +\Dt\sum_{i=1}^{n} (1-\ell \Dt)^{i-1}g_i \right).
\eeq
\end{lemma}
\begin{proof}  Let $\bar a_n =(1-\ell\Dt)^{n} a_n$. Simple calculation shows  that
\beqs
\frac{\bar a_n -\bar a_{n-1}}{\Dt} = (1-\ell\Dt)^{n-1}\left( \frac{a_n-a_{n-1}}{\Dt} -\ell a_n \right) \le (1-\ell \Dt)^{n-1} g_n.
\eeqs

Summation over $n$ leads to
\beqs
\frac{\bar a_n -\bar a_0}{\Dt} \le \sum_{i=1}^{n} (1-\ell\Dt)^{i-1} g_i,
\eeqs
and hence \eqref{Gineq} holds true.
\end{proof}

{\bf Notations:}   Let $L^2(\Omega)$ be the set of square integrable functions on $\Omega$ and $( L^2(\Omega))^d$ the space of $d$-dimensional vectors with all the components in $L^2(\Omega)$.  We denote $(\cdot, \cdot)$ the inner product in either $L^2(\Omega)$ or $(L^2(\Omega))^d$.  The notation $\norm {\cdot}$ means scalar norm $\norm{\cdot}_{L^2(\Omega)}$ or vector norm $\norm{\cdot}_{(L^2(\Omega))^d}$ and $ \norm{\cdot}_{L^p}=\norm{\cdot}_{L^p(\Omega)}$ represents the standard  Lebesgue norm.
Notation $\norm{\cdot}_{L^p(L^q)} =\norm{\cdot}_{L^p(0,T; L^q(\Omega))}, 1\le p,q<\infty$ means the mixed Lebesgue norm. 

For $1\le q\le \infty$ and $m$ any nonnegative integer, let
$
W^{m,q}(\Omega) = \big\{u\in L^q(\Omega), D^\alpha u\in L^q(\Omega), |\alpha|\le m \big\}
$
denote a Sobolev space endowed with the norm 
$
\norm{u}_{m,q} =
\left( \sum_{|i|\le m} \norm{D^i u}^q_{L^q(\Omega)} \right)^{\frac 1 q}.
$
Define $H^m(\Omega)= W^{m,2}(\Omega)$ with the norm $\norm{\cdot}_m =\norm{\cdot }_{m,2}$. 

Our estimates make use of coefficient-weighted norms. For some strictly positive,
bounded function, we denote the weighted $L^2$-norm by $\norm{f}_\omega$ by 
\beqs
\norm{f}_{\omega}^2 =\int_\Omega \omega |f|^2 dx,  
\eeqs
and if $0<\omega_* \le \omega (x) \le \omega^*$ thoughout $\Omega$, we have the equivalent 
\beqs
\sqrt{\omega_*} \norm{f}\le \norm{f}_{\omega} \le \sqrt{\omega^*}\norm{f}. 
\eeqs
We will also use weighted versions of Cauchy-Schwarz. With such a weight function
$\omega$, we can bound a standard inner product as
\beqs
\intd{f,g} \le \norm{f}_{\omega}\norm{g}_{\omega^{-1}}. 
\eeqs

Throughout this paper, we use short hand notations,  $\norm{\rho(t)} = \norm{ \rho(\cdot, t)}_{L^2(\Omega)},$  
 and $ \rho^0(\cdot) =  \rho(\cdot,0).$ The letters $C, C_0, C_1,C_2\ldots$ represent  positive generic constants.  Their values  depend on exponents, coefficients of polynomial  $g$,  the spatial dimension $d$ and domain $\Omega$, independent of the initial data and boundary data, size of mesh and time step. These constants may be different from place to place.

\section{A mixed finite element approximation}\label{MFEM}
In this section, we will present the mixed weak formulation of the Forchheimer equation. We consider the initial boundary value problem (IVBP) associated with \eqref{utporo}:
\beq\label{rho:eq}
\begin{cases}
 m   +K( |\nabla \rho|)\nabla \rho =0 & \text { in }  \Omega\times (0,T),\\
\phi \rho_t+ \nabla\cdot  m=f & \text { in }  \Omega\times (0,T),\\
m\cdot \nu= \psi(x,t) & \text { in  }\Gamma\times (0,T),\\
\rho(x,0) =\rho^0(x) & \text { in  }\Gamma\times (0,T).  
\end{cases}
\eeq 
where $\nu$ is a outer normal vector of boundary $\Gamma$, $\rho^0(x)$ and $\psi(x,t)$ are smooth functions .  

 Assume that $\phi(x)\in C^1(\Omega)$ and $0<\phi_*<\phi(x)<\phi^*$ for all $x\in\Omega$ then the system \eqref{rho:eq} reduces to the equation form 
 \beqs
 \rho_t - \nabla \cdot  (\phi^{-1} K(|\nabla \rho|)\nabla\rho )+\nabla \phi^{-1} K(|\nabla\rho|)\nabla\rho -f(x,t) =0. 
 \eeqs
 
This equation is a nonlinear degenerate parabolic equation as the density gradient approaches to infinity. The existence and theory of regularity for degenerate parabolic of this type is studied in \cite{IKO, ladyzhenskaiÍ¡a1988linear,Dibenedetto93, Ilyin2002, MR2566733}.   

Define  $\Ro=H^1(\Omega)$ and the space    
$$\M= H({\rm div}, \Omega)= \left\{ m \in  (L^2(\Omega))^d, \nabla \cdot m \in L^2(\Omega) \right\}$$
with the norm defined by $\norm{ m}_{\M}^2 = \norm{m}^2 + \norm{ \nabla \cdot m }^2.$

The variational formulation of \eqref{rho:eq} is defined as follows: 
Find $(m,\rho): I=(0,T)\to  \M \times \Ro$ such that 
\beq\label{weakform}
\begin{split}
 \intd{m, z}+\intd{K(|\nabla\rho|)\nabla\rho, z} =0,  \quad & z\in \M,\\
\intd{\phi\rho_t,r}-\intd{m,\nabla r }=\intd{f,r}-\intb{\psi,r},\quad   & r\in \Ro 
\end{split}
\eeq 
with $\rho(x,0)=\rho^0(x).$ 

\vspace{0.2cm}
 Let $\{\mathcal T_h\}_h$ be a family of quasiuniform triangulations of  $\overline{\Omega}$ with $h$ being the maximum diameter of the mesh elements. Let $\M_h,$ $\mathcal R_h $ be the space of discontinuous piecewise polynomials of degree $k\ge 0$ over $\mathcal T_h$.  Let $\M_h \times \Ro_h$ be the mixed element spaces approximating the space  $\M\times\Ro$.     
  
For density, we use the standard $L^2$-projection operator, see in \cite{Ciarlet78},  
$\pi: \Ro \to \Ro_h$,  satisfying
\beqs
( \pi \rho -\rho, r ) = 0 ,\quad \rho\in \Ro,  \forall r\in \Ro_h.
\eeqs 

This projection has well-known approximation properties, e.g.~\cite{BF91,JT81,BPS02}.
\begin{itemize}
\item[i.]  For all $\rho\in H^s(\Omega), s\in \{0,1\}$, there  is a positive constant $C_0$ such that
\beq\label{L2proj0}
\norm{\pi \rho}_{s}\le C_0 \norm{\rho}_{s}.   
\eeq

\item[ii.] There exists a positive constant $C_1$ such that for all $ \rho \in W^{s,q}(\Omega)$,
\beq\label{prjpi}
\norm{\pi \rho - \rho }_{0,q} \leq C_1 h^s \norm{\rho}_{s,q}, \quad   0\le s \le k+1, 1\le q \le \infty.
\eeq
When $q=2$, in short hand we write \eqref{prjpi} as   
\beqs
\norm{\pi \rho - \rho } \leq C_1 h^s \norm{\rho}_{s} . 
\eeqs
\end{itemize}

 The semidiscrete formulation of~\eqref{weakform} can read as follows: Find a pair $(m_h,\rho_h): I \to \M_h \times \Ro_h $ such that 
\beq\label{semidiscreteform}
\begin{split}
 \intd{m_h, z}+\intd{K(|\nabla\rho_h|)\nabla\rho_h, z} =0, \quad & \forall z\in \M_h,\\
\intd{\phi\rho_{h,t},r}-\intd{m_h,\nabla r }=\intd{f,r}-\intb{\psi,r} , \quad & \forall r\in \Ro_h 
\end{split}
\eeq 
with initial data $\rho_h^0=\pi \rho^0(x)$.

 Let $\{t_i\}_{i=1}^N$ be the uniform partition of $[0,T]$ with $t_i=i\Dt$, for time step $\Dt >0$. We define $\varphi^i =\varphi(\cdot, t_i) $.  The discrete time mixed finite element approximation to \eqref{weakform} is defined as follows:  
For given $\rho_h^0(x)=\pi \rho^0(x)$ and $\left\{f^i\right\}_{i=1}^N\in L^2(\Omega), \left\{\psi^i\right\}_{i=1}^N\in L^\infty(\bar\Omega)  $. Find a pair $ (m_h^i, \rho_h^i )$ in  $\M_h \times \Ro_h $, $i=1, 2,\ldots, N$  such that 
\beq\label{fullydiscreteform}
\begin{split}
  \intd{m_h^i, z}+\intd{K(|\nabla\rho_h^i|)\nabla\rho_h^i, z} =0,\quad &\forall z\in \M_h,\\
\intd{\phi\frac{ \rho_h^i - \rho_h^{i-1}}{t\Delta t },r}-\intd{m_h^i,\nabla r }=\intd{f^i,r}-\intb{\psi^i,r} , \quad &\forall r\in \Ro_h. 
\end{split}
\eeq

\section{Stability of semidiscrete approximation}\label{Bsec}
We study the  equations \eqref{weakform}, and \eqref{semidiscreteform} with fixed functions $g(s)$ in \eqref{eq1} and \eqref{eq2}. 
Therefore, the exponents $\alpha_i$ and coefficients $a_i$ are all fixed, and so are the functions $K(\xi)$, $H(\xi)$  in \eqref{Kdef}, \eqref{Hxi}.  

With the properties \eqref{i:ineq1}, \eqref{i:ineq2}, \eqref{i:ineq3}, the monotonicity \eqref{monotone0}, and by
classical theory of monotone operators in \cite{MR0259693,s97,z90}, the authors in \cite {HIKS1, galusinski2008nonlinear} proved the global existence and uniqueness of the weak solution of the equation \eqref{weakform}. For  the { \it priori } estimates, we assume that the weak solution is a sufficient regularity in both  $x$ and $t$ variables. 

\begin{theorem}\label{bound-lq} Let $(\rho_h, m_h)$ be a solution to the problem \eqref{semidiscreteform}. Then, there exists  a positive constant $C$ such that
\begin{align}
\label{res1a}
&\norm{\rho_h}_{L^\infty(I;L^2(\Omega))}^2 +\norm{\nabla \rho_h}_{L^\beta(I;L^\beta(\Omega))}^{\beta} \le  C\left(\norm{\rho^0}^2+ \mathcal A \right);&\\
\label{gradrho-m}
 &\norm{\rho_{h,t}}_{L^2(I;L^2(\Omega))}^2 +\norm{\nabla \rho_h}_{L^\infty(I;L^\beta(\Omega))}^{\beta} + \norm{m_h}_{L^\infty(I;L^2(\Omega))} \le C\mathcal B,\\
\text{ where } \qquad
\label{Adef}
&\mathcal A=1+\norm{\psi(t)}_{L^\infty(I; L^\infty(\Gamma))}^{\lambda } + \norm{f(t)}_{L^2(I;L^2(\Omega))}^2,\\
\label{Bdef}
&\mathcal B=\norm{\rho^0}^2+ \norm{\nabla \rho^0}_{0,\beta}^{\beta} +\norm{\psi(0)}_{L^2(\Gamma)} \norm{\rho^0 }_{L^2(\Gamma)}+ \norm{\psi_t}_{L^\infty(I,L^\infty(\Gamma))}^{\lambda}
+ \mathcal A.
\end{align} 
\end{theorem}
\begin{proof}
Choosing $z=\nabla \rho_h$ and $r=\rho_h$ in \eqref{semidiscreteform}, and adding the resultants, we find that  
\beq\label{term0}
\intd{\phi\rho_{h,t},\rho_h} +\intd{K(|\nabla\rho_h|)\nabla\rho_h, \nabla \rho_h} = \intd{f,\rho_h}-\intb{\psi,\rho_h}. 
\eeq
The second term of the LHS in \eqref{term0} is treated by using \eqref{i:ineq2} as follows
\beq\label{term1}
\intd{ K(|\nabla \rho_h|)|\nabla \rho_h|,\nabla \rho_h|} dx\ge c_0\int_\Omega (|\nabla \rho_h|^{\beta}-1)dx = c_0 \norm{\nabla \rho_h}_{0,\beta}^{\beta}-c_0|\Omega|. 
\eeq
We use Young's inequality and \eqref{sec2} to obtain 
\beq\label{term2}
\intd{f, \rho_h}- \intb{\psi, \rho_h }
\le \frac 12 \norm{\rho_h}^2 +\frac 12 \norm{f}^2+\norm{\psi}_{L^\infty(\Gamma)}\Big\{ C_1\norm{\rho_h} + \varep\norm{\nabla \rho_h}_{0,\beta}^{\beta} +C_1\varep^{-\frac 1{\beta-1}}   \Big\}.
\eeq
In view of \eqref{term0}, \eqref{term1} and \eqref{term2}, and selecting $\varep=\frac{c_0}{2}  ( \norm{\psi}_{L^\infty(\Gamma)}+1)^{-1}$, \eqref{term0} becomes
\beq \label{r:est}
\begin{split}
\frac{d }{dt}\norm{\rho_h}_{\phi}^2 +\frac{c_0} 2\norm{\nabla \rho_h}_{0,\beta}^{\beta}&\le C\norm{\psi}_{L^\infty(\Gamma)}\Big(\norm{\rho_h} + \left(\norm{\psi}_{L^\infty(\Gamma)}+1\right)^{\frac 1{\beta-1}}\Big) +C\Big(1+\norm{\rho_h}^2 +\norm{f}^2\Big)\\
&\le  C\norm{\rho_h}_{\phi}^2  + C\Big(1+\norm{\psi}_{L^\infty(\Gamma)}^{\frac{\beta}{\beta-1} } + \norm{f}^2\Big).
\end{split}
\eeq
Solving this differential inequality leads to 
\beq\label{res0}
\norm{\rho_h}_{\phi}^2 +\frac{c_0} 2 \int_0^t\norm{\nabla \rho_h}_{0,\beta}^{\beta}d\tau \le  \norm{\rho_{h}^0}_{\phi}^2+ C\int_0^t \Big(1+\norm{\psi}_{L^\infty(\Gamma)}^{\frac{\beta}{\beta-1} } + \norm{f}^2\Big) d\tau,
\eeq
which implies
\beqs
\norm{\rho_h}^2 +\int_0^t\norm{\nabla \rho_h}_{0,\beta}^{\beta}d\tau \le  C\norm{\rho_{h}^0}^2+ C\int_0^t \Big(1+\norm{\psi}_{L^\infty(\Gamma)}^{\frac{\beta}{\beta-1} } + \norm{f}^2\Big) d\tau.
\eeqs
Note that  
$
\norm{\rho^0_h}= \norm{\pi\rho^0}\le \norm{\rho^0}.
$ 
Thus inequality \eqref{res1a} holds.

(ii) Choosing $z=\nabla \rho_{h,t}$ and $r=\rho_{h,t}$ in \eqref{semidiscreteform}, and adding the resulting equations, we obtain    
\beq\label{Diffineq1}
 \norm{\rho_{h,t} }_{\phi}^2+\frac 1 2\ddt \int_\Omega H(x,t) dx =\intd{f,\rho_{h,t}}-\intb{\psi,\rho_{h,t}}=\intd{f,\rho_{h,t}}-\ddt \intb{\psi,\rho_h}+\intb{\psi_t,\rho_h}, 
\eeq
where $H(x,t) = H(|\nabla \rho_h(x,t)|)$.
Let
\beqs
\mathcal E(t) = \int_\Omega H(x,t) dx +\norm{\rho_h}_{\phi}^2 +2\intb{\psi, \rho_h}.
\eeqs
Adding \eqref{Diffineq1} and \eqref{r:est} gives  
 \beqs
\begin{aligned}
\norm{\rho_{h,t} }_{\phi}^2  +\frac{c_0}2\norm{\nabla \rho_h}_{0,\beta}^{\beta}+ \frac 12\ddt \mathcal E(t) \le  (f, \rho_{h,t})+\intb{ \psi_t, \rho_h} + C\norm{\rho_h}_{\phi}^2 +  C\Big(1+\norm{\psi}_{L^\infty(\Gamma)}^{\frac{\beta}{\beta-1} } + \norm{f}^2\Big).
\end{aligned}
\eeqs
Using \eqref{bnd-est} and Young's inequality leads to  
  \begin{align*}
\frac 12 \norm{\rho_{h,t} }_{\phi}^2 +\frac{c_0}{2} \norm{\nabla \rho_h}_{0,\beta}^{\beta}+ \frac 12 \ddt \mathcal E(t) &\le \frac 14 \left(\norm{\rho_h}^2+ \norm{\nabla \rho_h}_{0,\beta}^{\beta}\right)
+ C\left(1+ \norm{\psi_t}_{L^\infty(\Gamma)}^{^\lambda}  \right) \\
&+\frac 12 \norm{f}_{\phi^{-1}}^2  + C\norm{\rho_h}^2 + C\Big(1+\norm{\psi}_{L^\infty(\Gamma)}^{\lambda } + \norm{f}^2\Big).
\end{align*}
Integrating in time gives
\begin{align*}
\int_0^T (\norm{\rho_{h,t} }_{\phi}^2  +c_0 \norm{\nabla \rho_h}_{0,\beta}^{\beta}) dt+ \mathcal E(t) &\le C \int_0^T \left(\norm{\rho_h}^2+ \norm{\nabla \rho_h}_{0,\beta}^{\beta}\right) dt\\
 &+  C\int_0^T\Big(1+\norm{\psi}_{L^\infty(\Gamma)}^{\lambda }+\norm{\psi_t}_{L^\infty(\Gamma)}^{^\lambda} + \norm{f}^2\Big)dt+ \mathcal E(0).
\end{align*}
Then using \eqref{res1a}, we obtain
  \beq\label{Diffineq3}
\int_0^T (\norm{\rho_{h,t} }_{\phi}^2  +c_0 \norm{\nabla \rho_h}_{0,\beta}^{\beta}) dt+\int_\Omega H(x,t) dx +\norm{\rho_h}_{\phi}^2
\le -2\langle \psi, \rho_h\rangle+C\left(\norm{\psi_t}_{L^\infty(I,L^\infty(\Gamma))}^{\lambda}+ \mathcal A\right) +C\norm{\rho^0}^2+ \mathcal E(0).
\eeq
Applying \eqref{bnd-eps} to the first term of the RHS in \eqref{Diffineq3} and using the fact  that $c_3 (|\nabla\rho_h|^{\beta}-1) \le H(x,t)   \le 2c_2|\nabla\rho_h|^{\beta}$,  we have   
 \beq\label{Diffineq4}
\begin{aligned}
\int_0^T \norm{\rho_{h,t} }_{\phi}^2 dt +c_0 \int_0^T\norm{\nabla \rho_h}_{0,\beta}^{\beta} dt+  c_3\norm{\nabla \rho_h}_{0,\beta}^{\beta} +\norm{\rho_h}_{\phi}^2
\le 2\varep\left( \norm{\rho_h}^2+ \norm{\nabla \rho_h}_{0,\beta}^{\beta}\right)\\
+ C\left(\varep^{-1}\norm{\psi}_{L^\infty(\Gamma)}+ \varep^{-\frac 1{\beta-1}} \norm{\psi}_{L^\infty(\Gamma)}^{\lambda}  \right) +C\left(\norm{\psi_t}_{L^\infty(I,L^\infty(\Gamma))}^{\lambda}+ \mathcal A\right) +C\norm{\rho^0}^2+ \mathcal E(0).
\end{aligned}
\eeq
Then taking $\varep =\min\{c_3,1\}/4$ and using Young's inequality, \eqref{Diffineq4} leads to 
 \beq\label{ineq4}
\int_0^T \norm{\rho_{h,t} }_{\phi}^2 dt +c_0 \int_0^T \norm{\nabla \rho_h}_{0,\beta}^{\beta} dt+ \frac{c_3}2 \norm{\nabla \rho_h}_{0,\beta}^{\beta} +\frac 12 \norm{\rho_h}_{\phi}^2 \le C\left(\norm{\psi_t}_{L^\infty(I,L^\infty(\Gamma))}^{\lambda}+ \mathcal A\right) +C\norm{\rho^0}^2+ \mathcal E(0).
\eeq
Note that 
\beq\label{E0}
 \mathcal E(0) \le C\left(\norm{\rho^0}^2+ \norm{\nabla \rho^0}_{0,\beta}^{\beta} +\norm{\psi(0)}_{L^2(\Gamma)} \norm{\rho^0 }_{L^2(\Gamma)}\right).
\eeq

Putting estimates \eqref{ineq4} and \eqref{E0} together, we obtain the first part of \eqref{gradrho-m}.    

Now choosing $z=m_h$ in the first equation of \eqref{semidiscreteform} gives
$
\norm{m_h}^2 +\intd{K(|\nabla\rho_h|)\nabla\rho_h, m_h} =0,
$
which leads to 
\beq\label{mm}
\norm{m_h} \le \norm { K(|\nabla\rho_h|)\nabla\rho_h} \le C\Big( \int_\Omega K(|\nabla\rho_h|)|\nabla\rho_h|^2 dx \Big)^{1/2} \le C \norm{\nabla \rho_h}_{0,\beta}^{\beta}.
\eeq
This, \eqref{ineq4} and \eqref{E0} imply the second part of \eqref{gradrho-m}. The proof is complete. 
\end{proof}

\begin{remark}
 The equation \eqref{semidiscreteform} can be interpreted as the finite system of ordinary differential equations in the coefficients of $(m_h, \rho_h)$ with respect to basis of $\M_h\times \Ro_h$. The stability estimates \eqref{res1a} and \eqref{gradrho-m} suffice to establish the local existence of $(m_h(t),\rho_h(t))$ for all $t\in (0,T).$   
\end{remark}

The uniqueness of the approximation solution comes from the monotonicity of  the operator, see in \cite{HI1}. In fact, assume that for $i=1,2$, $\{m_{h,i}, \rho_{h,i}\}$ are two solutions of \eqref{semidiscreteform}.  Let 
  $  \mu= m_{h,1} - m_{h,2},  \varrho = \rho_{h,1} - \rho_{h,2}.  $  Then
  \beq\label{DiffErr}
\begin{split}
 \intd{\mu, z}+\intd{K(|\nabla\rho_{h,1}|)\nabla\rho_{h,1}-K(|\nabla\rho_{h,2}|)\nabla\rho_{h,2}, z}  =0,\quad z\in \M_h,\\
\intd{\phi \varrho_t,r}-\intd{\mu,\nabla r }=0, \quad r\in \Ro_h. 
\end{split}
\eeq

  It is easily to see that with $z=\nabla \varrho$ and $r=\varrho$ in \eqref{DiffErr}
    \beqs
\intd{\phi \varrho_t,\varrho} +\intd{K(|\nabla\rho_{h,1}|)\nabla\rho_{h,1}-K(|\nabla\rho_{h,2}|)\nabla\rho_{h,2}, \nabla \varrho} =0. 
\eeqs
Thanks to the monotonicity \eqref{monotone0}, we see that
\beq\label{a}
\frac 1 2\ddt \norm{\varrho}_{\phi}^2 +C \int_\Omega \Big(\int_0^1 K(|\gamma(s)| ds )\Big) |\nabla \varrho|^2 dx  \le 0. 
\eeq 

Choosing $z=\mu$ in the first equation of \eqref{DiffErr}  and using the fact  the function $K(\cdot)\le a_0^{-1}$ lead to
\beq\label{b}
\norm{\mu}^2 \le C \int_\Omega \Big(\int_0^1 K(|\gamma(s)| ds )\Big)^2 |\nabla \varrho|^2 dx \le  \int_\Omega \Big(\int_0^1 K(|\gamma(s)| ds )\Big) |\nabla \varrho|^2 dx .  
\eeq
Putting \eqref{b} into \eqref{a} gives  
\beq
\frac 1 2\ddt \norm{\varrho}_{\phi}^2 +\norm{\mu}^2 \le 0.
\eeq
This implies  
$
\norm{\varrho}_{\phi}^2 +\norm{\mu}^2 \le C\norm{\varrho(0)}_{\phi}^2 =0. 
$
Hence $\varrho=0$ and $\mu=0$ a.e.    
\begin{theorem} Let $0<t_*<T$. Suppose $(m_h,\rho_h)$ be a solution of the problem \eqref{semidiscreteform}. There exists  a positive constant $C$ such that for all $t\in[t_*, T]$, 
\beq\label{Est4rhot}
 \norm{\rho_{h,t}(t)}^2 \le  C (1+t_*^{-1}) \mathcal B 
 +C\int_{0}^T \left(1+ \norm{\psi_t(\tau)}_{L^\infty(\Gamma)}\right)^{2\lambda}\left(1+\norm{f_t(\tau)}\right)^2 d\tau,
  \eeq
  where $\Bb$ is defined in \eqref{Bdef}. 
\end{theorem}
\begin{proof}
Taking time derivative \eqref{semidiscreteform}, choosing $z=\nabla \rho_{h,t}$ and $r=\rho_{h,t}$ in \eqref{semidiscreteform}, we obtain the equations 
\beq
\begin{split}
 \intd{m_{h,t}, \nabla \rho_{h,t}}+\intd{ K'(|\nabla\rho_h|) \frac{\nabla\rho_h\cdot \nabla\rho_{h,t}}{|\nabla\rho_h|} \nabla\rho_h  + K(|\nabla\rho_h|)\nabla\rho_{h,t}, \nabla \rho_{h,t}} =0,\\
\intd{\phi\rho_{h,tt},\rho_{h,t}}-\intd{m_{h,t},\nabla \rho_{h,t} }=\intd{f_t,\rho_{h,t}}-\intb{\psi_t,\rho_{h,t}} . 
\end{split}
\eeq 
Adding these equations yields 
    \beq\label{first-ineq}
  \frac{1}{2}\ddt \norm{\rho_{h,t}}^2 +\norm{K^{1/2}(|\nabla \rho_h|)\nabla \rho_{h,t}}^2 =- \intd{K'(|\nabla \rho_h|)\frac{\nabla \rho_{h}\cdot \nabla \rho_{h,t}}{|\nabla \rho_h|}\nabla \rho_h ,\nabla \rho_{h,t}   } + \intd{ f_t, \rho_{h,t} }-\intb{\psi_t, \rho_{h,t}}.
  \eeq
  According to \eqref{i:ineq3}, 
  \beq\label{I1}
  \left|- \intd{K'(|\nabla \rho_h|)\frac{\nabla \rho_{h}\cdot \nabla \rho_{h,t}}{|\nabla \rho_h|}\nabla \rho_h ,\nabla \rho_{h,t}   }\right|\le a\norm{K^{1/2}(|\nabla \rho_h|)\nabla \rho_{h,t}}^2.
  \eeq
  The inequality \eqref{first-ineq} deduces  to  
  \beq\label{midstep0}
  \frac{1}{2}\ddt \norm{\rho_{h,t}}^2 +(1-a)\norm{K(|\nabla \rho_h|)\nabla \rho_{h,t}}^2\le \intd{ f_t, \rho_{h,t} }-\intb{\psi_t, \rho_{h,t}}.
  \eeq
 In virtue of Young's inequality,  for all $\varep>0$
\beq\label{I2}
( f_t,\rho_{h,t} ) \le \varep\norm{\rho_{h,t}}^2 + C\varep^{-1}\norm{f_t}^2.
\eeq
Using Trace Theorem we obtain,
  \beq\label{I30}
  \left|\intb{\psi_t, \rho_{h,t}}\right|\le \norm{\psi_t}_{L^\infty(\Gamma)}\left[ \intd{|\rho_{h,t}| ,1}+\intd{|\nabla \rho_{h,t}|,1 }\right].
  \eeq
  Again Young's inequality gives  
   \beq\label{Yt2}
   (|\rho_{h,t}|,1) \le \varep\norm{\rho_{h,t}}^2+C\varep^{-1}, 
 \quad  
 (|\nabla \rho_{h,t}|,1) \le  \varep_1 ( K(|\nabla \rho_h|)|\nabla \rho_{h,t}|^2,1)+C\varep_1^{-1} ( K^{-1}(|\nabla \rho_h|),1). 
 \eeq
By using \eqref{i:ineq1} and $(1+x)^a \le 1+x^a, x\ge 0$ imply   
 \beq\label{Yt1}
(|\nabla \rho_{h,t}|,1) 
   \le  \varep_1 \norm{K^{1/2}(|\nabla \rho_h|)\nabla \rho_{h,t}}^2+C\varep_1^{-1}(1+\norm{\nabla \rho_h}_{0,\beta}^{a}). 
 \eeq
 In view of \eqref{Yt2} and \eqref{Yt1}, \eqref{I30} becomes   
 \beq\label{I3}
 \begin{split}
\left|\intb{\psi_t, \rho_{h,t}}\right|\le\norm{\psi_t}_{L^\infty(\Gamma)} \Big\{\varep\norm{\rho_{h,t}}^2+C\varep^{-1}+ \varep_1\norm{ K^{1/2}(|\nabla \rho_h|)\nabla \rho_{h,t} }^2+ C\varep_1^{-1}(1+\norm{\nabla \rho_h}_{0,\beta}^{a})  \Big\}.
\end{split}
\eeq
 It follows from \eqref{midstep0}, \eqref{I2} and \eqref{I3} that  
\beq\label{bb}
\begin{split}
  \frac{1}{2}\ddt \norm{\rho_{h,t}}^2 &+(1-a)\norm{K^{1/2}(|\nabla \rho_h|)\nabla \rho_{h,t}}^2  \le \norm{\psi_t}_{L^\infty(\Gamma)} \Big\{\varep\norm{\rho_{h,t}}^2+C\varep^{-1}\\
  & + \varep_1\norm{ K^{1/2}(|\nabla \rho_h|)\nabla \rho_{h,t} }^2 +C\varep_1^{-1}(1+ \norm{\nabla\rho_h}_{0,\beta}^{a})  \Big\}+\varep\norm{\rho_{h,t}}^2 +C\varep^{-1}\norm{f_t}^2.
  \end{split}
 \eeq
 
 Selecting $\varep_1=(1-a)\varep=\frac{1-a} 2(\|\psi_t\|_{L^\infty(\Gamma)}+1)^{-1}$ yields   
    \begin{align*} 
\ddt \norm{\rho_{h,t}}^2 +(1-a)\norm{ K(|\nabla\rho_h|)\nabla \rho_{h,t} }^2&\le C\norm{\rho_{h,t}}^2 +C \norm{\psi_t}_{L^\infty(\Gamma)}(\norm{\psi_t}_{L^\infty(\Gamma)} +1)\norm{\nabla \rho_h}_{0,\beta}^{a}\\
&\quad+C\left(\norm{\psi_t}_{L^\infty(\Gamma)}+1 \right) \left(\norm{\psi_t}_{L^\infty(\Gamma)} +\norm{f_t}^2\right)\\
&\le \norm{\rho_{h,t}}^2 + \norm{\nabla \rho_h}_{0,\beta}^{\beta}+C Z(t),
\end{align*}
where 
$Z(t) =\left(\norm{\psi_t}_{L^\infty(\Gamma)}+1 \right) \left(\norm{\psi_t}_{L^\infty(\Gamma)} +\norm{f_t}^2\right)
 +\left(\norm{\psi_t}_{L^\infty(\Gamma)}(1+ \norm{\psi_t}_{L^\infty(\Gamma)})\right)^\lambda.
$

  For $t\ge t_* \ge t'>0$. Ignoring the the nonnegative term in the LHS of the above inequality, integrating from $t'$ to $t$ and then integrating in $t'$ from $0$ to $t_*$, we find that    
\beqs
 t_* \norm{\rho_{h,t}}^2  \le  \int_0^{t_*}\norm{\rho_{h,t}(t')}^2 dt'+t_*\int_0^t \big(\norm{\rho_{h,t}}^2 +\norm{\nabla \rho_h}_{0,\beta}^{\beta}\big)d\tau
 +Ct_*\int_{0}^t Z(\tau)d\tau.
  \eeqs
By virtue of \eqref{gradrho-m},   
\beqs
\int_0^t\big(\norm{\rho_{h,t}}^2+\norm{\nabla \rho_h}_{0,\beta}^{\beta}\big)dt\le C \mathcal B,\quad \int_0^{t_*}\norm{\rho_{h,t}(t')}^2 dt' \le C \mathcal B.
\eeqs    
Therefore, 
\beq\label{al1}
 t_* \norm{\rho_{h,t}}^2 \le  C \mathcal B +C \mathcal B t_*
 +Ct_*\int_{0}^t Z(\tau) d\tau.
  \eeq
We estimate $Z$-term by
\beq\label{al2}
\begin{split}
Z(t)\le \left(\norm{\psi_t}_{L^\infty(\Gamma)}+1 \right)^{2\lambda}(1+ \norm{f_t}^2) 
 +\left(1+ \norm{\psi_t}_{L^\infty(\Gamma)}\right)^{2\lambda}
 \le 2\left(1+ \norm{\psi_t}_{L^\infty(\Gamma)}\right)^{2\lambda}\left(1+\norm{f_t}^2\right).
\end{split}
\eeq
The inequality \eqref{Est4rhot} follows from \eqref{al1} and \eqref{al2}. The proof is complete.  
\end{proof}

In the same manner to the problem~\eqref{weakform}, we have as the following:  
 
\begin{theorem}\label{Est4sol} Let $0<t_*<T$. Suppose $(\rho, m)$ be a solution of the problem \eqref{weakform}. Then, there exists  a positive constant $C$ such that
\begin{align}
\label{Est-rho}
&\norm{\rho}_{L^\infty(I;L^2(\Omega))}^2 +\norm{\nabla \rho}_{L^\beta(I;L^\beta(\Omega))}^{\beta} \le  C\left(\norm{\rho^0}^2+ \mathcal A \right);&\\
\label{Est-m}
 &\norm{\rho_{t}}_{L^\infty(I;L^2(\Omega))}^2 +\norm{\nabla \rho}_{L^\infty(I;L^\beta(\Omega))}^{\beta} + \norm{m}_{L^\infty(I;L^2(\Omega))} \le C\mathcal B;\\
 \label{Est-rhot}
 &\norm{\rho_{t}(t)}^2 \le  C (1+t_*^{-1}) \mathcal B 
 +C\int_{0}^T \left(1+ \norm{\psi_t(\tau)}_{L^\infty(\Gamma)}\right)^{2\lambda}\left(1+\norm{f_t(\tau)}\right)^2 d\tau \quad \forall t\in [t_*, T],
\end{align} 
where $\mathcal A, \mathcal B$ are defined in \eqref{Adef} and \eqref{Bdef}.
\end{theorem}

\section{Dependence of solutions on parameters} \label{dependence-sec}
In this section, we study the dependence of the solution on the coefficients of  Forchheimer polynomial $g(s)$ in \eqref{eq2}.
Let $N\ge 1$, the exponent vector $\vec \alpha=(0,\alpha_1,\ldots,\alpha_N)$ and the boundary data $\psi(x,t)$  be fixed.  Let ${\bf D}$ be a compact subset of $\{\vec a=(a_0,a_1,\ldots,a_N):a_0,a_N>0, a_1,\ldots,a_{N-1}\ge 0\}$.
Set 
$\hat \chi({\bf D})=\max\{\chi(\vec a):\vec a\in {\bf D}\}.$
Then $\hat \chi ({\bf D})$ is a number in $[1,\infty)$.

Let $g_1(s)=g(s,\vec a_1)$ and $g_2(s)=g(s,\vec a_2)$ be two functions of class FP($N,\vec \alpha$), where  $\vec a_1$ and $\vec a_2$ belong to  ${\bf D}$. Let $\rho_1=\rho_1(x,t;\vec a_1)$, $\rho_2=\rho_2(x,t;\vec a_2)$  be the two solutions of \eqref{weakform} respective to  $K(\xi,\vec a_1)$, $K(\xi,\vec a_2)$  with the same boundary data $\psi$ and initial data $\rho^0$. We will estimate $\norm{\rho_1-\rho_2}$, $\norm{m_1 -m_2}$ in the term of $|\vec a_1-\vec a_2|$. 

Let $\varrho=\rho_1-\rho_2$, $\mu=m_1 - m_2$. Then 
\beq\label{ss1}
\begin{split}
 \intd{\mu, z}+\intd{K(|\nabla\rho_1|, \vec a_1)\nabla\rho_1 - K(|\nabla\rho_2|,\vec a_2)\nabla\rho_2 , z} =0,\quad &\forall z\in \M,\\
\intd{\phi\varrho_{t},r}-\intd{\mu,\nabla r }=0 , \quad &\forall r\in \Ro.  
\end{split}
\eeq 

\begin{theorem}\label{DepCoeff} Given $0<t_*<T$. Let $(\rho_i, m_i), i=1,2$ be two solutions to problem \eqref{weakform} corresponding to vector coefficients  $\vec a_i$ of Forchheimer polynominal $g(s,\vec a_i)$ in \eqref{eq2}. There exists a constant positive constant $C$ independent of $|\vec a_1- \vec a_2|$ such that    
\beq\label{ssc2}
\norm{\rho_1-\rho_2}_{L^\infty(I;L^2(\Omega))} +\norm{m_1-m_2}_{L^\infty(t_*,T;L^2(\Omega))}^2 \le C |\vec a_1-\vec a_2| .
\eeq
\end{theorem}
\begin{proof} 
Choosing $z=\nabla \varrho$ and $r=\varrho$ in \eqref{ss1}, and adding the resulting equations, we find that
\beq\label{Diffeq}
\frac 12\frac{d}{dt}\norm{\varrho}_{\phi}^2+ (K(|\nabla \rho_1|,\vec a_1)\nabla \rho_1-K(|\nabla \rho_2|,\vec a_2)\nabla \rho_2, \nabla \varrho)=0.
\eeq
According to \eqref{quasimonotone}, 
\beq\label{keyEst}
\begin{aligned}
\frac 12\frac{d}{dt}\norm{\varrho}_{\phi}^2&
\le -(\beta-1)\int_\Omega \Big(\int_0^1 K(|\gamma(s)|,\vec b(s)) ds\Big) |\nabla \varrho|^2  dx\\
&\quad +C  |\vec a_1 -\vec a_2|\int_\Omega \Big(\int_0^1 K(|\gamma(s)|,\vec b(s)) ds\Big) (|\nabla \rho_1|\vee |\nabla \rho_2|)|\nabla \varrho|   dx\\
&\le -\frac{\beta-1}{2}\int_\Omega \Big(\int_0^1 K(|\gamma(s)|,\vec b(s)) ds\Big) |\nabla \varrho|^2 dx\\
&\quad+C |\vec a_1 -\vec a_2|^2\int_\Omega  \Big(\int_0^1 K(|\gamma(s)|,\vec b(s)) ds\Big)( |\nabla \rho_1|\vee |\nabla \rho_2|)^2 dx.
\end{aligned}
\eeq
Using Poincare's inequality and H\"older's inequality, we obtain  
\beq\label{A}
\norm{\varrho }^2 \le C_p\norm{\nabla \varrho }_{0,\beta} ^2\le C_p \Big(\int_\Omega \Big(\int_0^1 K(|\gamma(s)|,\vec b(s)) ds\Big)|\nabla \varrho |^2dx\Big)\Big(\int_\Omega  \Big(\int_0^1 K(|\gamma(s)|,\vec b(s)) ds\Big)^{-\frac{\beta}{2-\beta}} dx\Big)^\frac{2-\beta}{\beta},
\eeq
which implies that
\beqs
\begin{split}
\int_\Omega \Big(\int_0^1 K(|\gamma(s)|,\vec b(s)) ds\Big)|\nabla \varrho |^2 dx \ge C_p^{-1}\norm{\varrho }^2\Big(\int_\Omega \Big(\int_0^1 K(|\gamma(s)|,\vec b(s)) ds\Big)^{-\frac{\beta}{2-\beta}} dx\Big)^{-\frac{2-\beta}{\beta}}.
\end{split}
\eeqs
Hence
\beq\label{Rineq}
\begin{split}
\frac 12\frac{d}{dt}\norm{\varrho}_{\phi}^2
&\le -C_*\norm{\varrho}_{\phi}^2\Big(\int_\Omega \Big(\int_0^1 K(|\gamma(s)|,\vec b(s)) ds\Big)^{-\frac{\beta}{a}} dx\Big)^{-\frac{a}{\beta}}\\
&+C |\vec a_1 -\vec a_2|^2\int_\Omega \Big(\int_0^1 K(|\gamma(s)|,\vec b(s)) ds\Big)(|\nabla \rho_1| \vee |\nabla \rho_2|)^2  dx,
\end{split}
\eeq
where $C_* = \frac{\beta-1}{2}C_p^{-1}\phi^*.$
Define 
\beqs
\Lambda (t) = \int_\Omega \Big(\int_0^1 K(|\gamma(s)|,\vec b(s)) ds\Big)^{-\frac{\beta}{a}} dx.
\eeqs

 Applying Gronwall Lemma  to \eqref{Rineq}, and using the fact that $\varrho (0)=0$, we obtain \eqref{ssc2}.
 \beq\label{sq}
\norm{\varrho}_{\phi}^2
\le  C |\vec a_1 -\vec a_2|^2  \int_0^t \Big(e^{-\int_s^t \Lambda(\tau)^{\gamma}  d\tau}\int_\Omega \Big(\int_0^1 K(|\gamma(s)|,\vec b(s)) ds\Big)(|\nabla \rho_1| \vee |\nabla \rho_2|)^2  dx\Big) dt.
\eeq
The last thing is to estimate 
\beq\label{ss}
\begin{aligned}
\Big(\int_0^1 K(|\gamma(t)|,\vec{b}(t)) dt\Big)(|\nabla\rho_1|\vee |\nabla\rho_2|)^2  &\le  K(|\nabla\rho_1|\vee|\nabla\rho_2| ,\vec a_1 \wedge\vec a_2) \Big)(|\nabla\rho_1|\vee |\nabla\rho_2|)^2\\
&\le C (|\nabla\rho_1|\vee |\nabla\rho_2|)^{\beta}
\le C (|\nabla\rho_1|^{\beta}+|\nabla\rho_2|^{\beta}). 
\end{aligned}
\eeq

Substituting \eqref{ss} to \eqref{sq} leads to    
\beq\label{ro}
\norm{\varrho}_{\phi}^2
\le C |\vec a_1 -\vec a_2|^2 \int_0^t   \big( \norm{\nabla \rho_1}_{0,\beta}^{\beta}+ \norm{\nabla \rho_2}_{0,\beta}^{\beta} \big) d\tau.
\eeq
We estimate the RHS of \eqref{ro} using \eqref{Est-m} to obtain the estimate for the first term in \eqref{ssc2}. 

To the second estimate in \eqref{ssc2}, we rewrite \eqref{keyEst} as follows
\beq\label{bs1}
\begin{split}
\int_\Omega \Big(\int_0^1 K(|\gamma(s)|,\vec b(s)) ds\Big) |\nabla \varrho|^2  dx&\le C\Big[\norm{\varrho _t}_{\phi}\norm{\varrho}_{\phi} \\
&+ |\vec a_1 -\vec a_2|^2\int_\Omega \Big(\int_0^1 K(|\gamma(t)|,\vec{b}(t)) dt\Big) (|\nabla\rho_1|\vee |\nabla\rho_2|)^2   dx\Big].
\end{split}
\eeq
By the mean of the triangle inequality and \eqref{Est-rhot} yield  
\beq\label{ybs2}
\norm{\varrho _t}_{\phi}\le \norm{\rho_{1,t}}_{\phi} +\norm{\rho_{2,t}}_{\phi} \le \phi^*(\norm{\rho_{1,t}} +\norm{\rho_{2,t}})\le C_{\rho^0, \psi}.
\eeq
Next plugging \eqref{ss}, \eqref{ro}, \eqref{ybs2} into \eqref{bs1}, we obtain 
\beq\label{aa}
\int_\Omega \Big(\int_0^1 K(|\gamma(s)|,\vec b(s)) ds\Big) |\nabla \varrho|^2  dx \le C\Big[ |\vec a_1 -\vec a_2|\big(\norm{\nabla\rho_1}_{0,\beta}^\beta +\norm{\nabla\rho_2}_{0,\beta}^\beta\big)^{\frac 1 2} + |\vec a_1 -\vec a_2|^2\big(\norm{\nabla\rho_1}_{0,\beta}^\beta +\norm{\nabla\rho_2}_{0,\beta}^\beta\big)\Big].
\eeq
Taking $z=\mu$ in \eqref{ss1} gives
\beq\label{mueq}
 \norm{\mu}^2+(K(|\nabla \rho_1|,\vec a_1)\nabla \rho_1-K(|\nabla \rho_2|,\vec a_2)\nabla \rho_2, \mu) =0.
\eeq
Applying H\"older's inequality to \eqref{mueq} and  then \eqref{Umono} gives    
\begin{multline}\label{bbb}
\norm{\mu}^2 \le \int_\Omega \left( K(|\nabla\rho_1|, \vec a_1)\nabla\rho_1 - K(|\nabla\rho_2|, \vec a_2)\nabla\rho_2 \right)^2 dx
\le 2(1+a)^2 \int_\Omega \Big(\int_0^1  K(|\gamma(t)|,\vec{b}(t)) dt \Big)^2 |\nabla\varrho|^2 dx\\
+ C |\vec a_1 -\vec a_2|^2\int_\Omega \Big(\int_0^1 K(|\gamma(t)|,\vec{b}(t)) dt\Big)^2  (|\nabla\rho_1|\vee |\nabla\rho_2|)^2   dx. 
\end{multline}
Thanks to the upper boundedness of $K(\cdot)$, \eqref{aa} and \eqref{ss}
\beqs
\norm{\mu}^2\le C\left[  |\vec a_1 -\vec a_2|\big(\norm{\nabla\rho_1}_{0,\beta}^\beta +\norm{\nabla\rho_2}_{0,\beta}^\beta\big)^{\frac 1 2} + |\vec a_1 -\vec a_2|^2\big(\norm{\nabla\rho_1}_{0,\beta}^\beta +\norm{\nabla\rho_2}_{0,\beta}^\beta\big)\right].  
\eeqs
Then, we use \eqref{Est-m} to bound $\norm{\nabla\rho_i}_{0,\beta}, i=1,2$ to obtain   
\beqs
\norm{\mu}^2\le C\left( |\vec a_1 -\vec a_2| + |\vec a_1 -\vec a_2|^2\right).   
\eeqs
This proves the estimate for the second term in \eqref{ssc2}. The proof is complete. 
\end{proof}
\section {Error estimates for semidiscrete approximation}\label{err-sec}
In this section, we will give the error estimate between the analytical solution and approximate solution. 
  We define the new variables:
  \beq\label{spliterr}
  \begin{split}
  m-m_h = m-\pi m - ( m_h- \pi m )=\eta - \zeta_h,\\ 
  \rho-\rho_h = \rho-\pi \rho - (\rho_h  - \pi\rho)= \theta -\vartheta_h.      
  \end{split}
  \eeq 

\begin{theorem}\label{err-thm} Given $0<t_*<T$. Let $(\rho,m)$ be the solution of \eqref{weakform} and $(\rho_h, m_h)$ be the solution of \eqref{semidiscreteform}. Suppose that $(\rho, m)\in L^\infty(I; H^{k+1}(\Omega))\times (L^2(I; L^2(\Omega)))^d$ and  $\rho_t\in L^2(I;H^{k+1}(\Omega) )$. Then there exists a positive constant $C$ independence of $h$ such that    
 \beq\label{err20}
  \norm {\rho -\rho_h}_{L^\infty (I;L^2(\Omega))}+ \norm {m -m_h}_{L^2 (I;L^2(\Omega))}+\norm {m -m_h}_{L^\infty (t_*,T;L^2(\Omega))}^2 \le Ch^k. 
 \eeq
  \end{theorem}
\begin{proof}
With the error written in \eqref{spliterr}, it sufficies, in the view of \eqref{prjpi}, to bound $\vartheta_h, \zeta_h$.     
Subtracting the weak equations and its finite approximation, we obtain the error equations              
\beq\label{ErrEq1}
\begin{split}
\intd{m-m_h, z}+\intd{K(|\nabla\rho|)\nabla\rho - K(|\nabla\rho_h|)\nabla\rho_h, z} =0,\quad \forall z\in \M_h,\\
\intd{\phi(\rho_t-\rho_{h,t}),r}-\intd{m-m_h,\nabla r }=0 , \quad \forall r\in \Ro_h.  
\end{split}
\eeq 
We rewrite these equations as follows:
\beq\label{ErrEqPrj}
\begin{split}
\intd{ m-m_h, z}+\intd{K(|\nabla\rho|)\nabla\rho - K(|\nabla\rho_h|)\nabla\rho_h, z} =0, \quad \forall z\in \M_h\\
\intd{\phi\vartheta_{h,t} ,r}+\intd{m-m_h,\nabla r }=\intd{\phi\theta_t ,r},\quad \forall r\in \Ro_h.
\end{split}
\eeq
Selecting $z=-\nabla \vartheta_h$ and $r=\vartheta_h$, and adding two above equations gives 
\beqs
\intd{\phi\vartheta_{h,t} ,\vartheta_h} - \intd{K(|\nabla\rho|)\nabla\rho - K(|\nabla\rho_h|)\nabla\rho_h, \nabla \vartheta_h }= \intd{\phi\theta_t ,\vartheta_h},
\eeqs  
or 
\beq\label{midres0}
\frac1 2 \ddt \norm{\vartheta_h}_{\phi}^2+ \intd{K(|\nabla\rho|)\nabla\rho - K(|\nabla\rho_h|)\nabla\rho_h, \nabla \rho -\nabla\rho_h }= \intd{K(|\nabla\rho|)\nabla\rho - K(|\nabla\rho_h|)\nabla\rho_h, \nabla \theta }+ \intd{\phi\theta_t ,\vartheta_h}.
\eeq  
By the monotonicity of $K(\cdot)$ in \eqref{monotone0}, 
 \beq\label{Kl}
 (K(|\nabla \rho|)\nabla \rho -K(|\nabla \rho_h|)\nabla \rho_h ,\nabla\rho -\nabla\rho_h)  \ge (\beta-1) \int_\Omega \Big(\int_0^1 K(|\gamma(s)|) ds\Big)|\nabla\rho -\nabla \rho_h|^2dx . 
    \eeq
   
By Young's inequality, for $\varep_0>0$  
\beq\label{r1}
(\phi\theta_t, \vartheta_h) \le C\varep_0^{-1} \norm{\theta_t}_{\phi}^2 +\varep_0 \norm{\vartheta_h}_{\phi}^2.
\eeq
Applying \eqref{Lipchitz}, Young's inequalities and the upper boundedness of $K(\cdot)$,  we obtain
\beq\label{Kr}
\begin{split}
(K(|\nabla \rho|)\nabla \rho -K(|\nabla \rho_h|)\nabla \rho_h ,\nabla \theta)
&\le (1+a)\int_\Omega  \Big(\int_0^1 K(|\gamma(s)|) ds\Big)|\nabla \rho-\nabla \rho_h||\nabla \theta| dx\\
&\le \frac{\beta-1} 2 \int_\Omega  \Big(\int_0^1 K(|\gamma(s)|) ds\Big)|\nabla \rho-\nabla \rho_h|^2 dx +C \norm{\nabla \theta}^2 .
\end{split}
\eeq
Combining \eqref{Kl}, \eqref{r1}, \eqref{Kr} and \eqref{midres0} gives  
\begin{align*}
\frac 12\ddt\norm{\vartheta_h}_{\phi}^2+\frac{\beta-1} 2 \int_\Omega  \Big(\int_0^1 K(|\gamma(s)|) ds\Big)|\nabla \rho-\nabla \rho_h|^2 dx\le C \norm{\nabla \theta}^2+C\varep_0^{-1}\norm{ \theta_t}_{\phi}^2 +\varep_0\norm{\vartheta_h}_{\phi}^2.
\end{align*}
Integrating in time from $0$ to $T$ and then taking sup-norm in time of the resultant shows that 
\beqs
 \sup_{t\in[0,T]}\norm{\vartheta_h}_{\phi}^2+  \int_0^T\int_\Omega \Big(\int_0^1 K(|\gamma(s)|) ds\Big)|\nabla \rho-\nabla \rho_h|^2dx dt 
\le C\norm{\nabla \theta}_{L^\infty(I,L^2)}^2 +C\varep_0^{-1}\norm{\theta_t}_{L^\infty(I,L^2)}^2 +\varep_0 T \sup_{t\in[0,T]}\norm{\vartheta_h}_{\phi}^2.
\eeqs
Selecting $\varep_0 =1/(2T),$ we find that 
\beq\label{errl2}
\sup_{t\in[0,T]}\norm{\vartheta_h}_{\phi}^2 +  \int_0^T\int_\Omega \Big(\int_0^1 K(|\gamma(s)|) ds\Big)|\nabla \rho-\nabla \rho_h|^2 dx dt\le C \norm{\nabla \theta}_{L^\infty(I,L^2)}^2 +CT\norm{\theta_t}_{L^\infty(I,L^2)}^2.
\eeq

Using $L^2$-projection and choose $z=\zeta_h $ in the first equation in \eqref{ErrEq1} yields
\beqs
\norm{\zeta_h}^2 = - \intd{K(|\nabla\rho|)\nabla\rho - K(|\nabla\rho_h|)\nabla\rho_h, \zeta_h} \le C \int_\Omega \Big(\int_0^1 K(|\gamma(s)|) ds\Big) |\nabla \rho-\nabla \rho_h || \zeta_h| dx.  
\eeqs
It follows from Cauchy's inequality and the upper boundedness of the function $K(\cdot)$ that 
\beq\label{mL2}
 \norm{\zeta_h}^2 dt \le C  \int_\Omega \Big(\int_0^1 K(|\gamma(s)|) ds\Big) |\nabla \rho-\nabla \rho_h |^2 dx dt.
\eeq   
Putting \eqref{errl2} and \eqref{mL2} together, using equivalent norm, we see that 
\beq\label{core0}
\norm{\vartheta_h}_{L^\infty(I,L^2)}^2 +  \norm{\zeta_h}_{L^2(I,L^2)}^2  \le C \norm{\nabla \theta}_{L^\infty(I,L^2)}^2 +CT\norm{\theta_t}_{L^\infty(I,L^2)}^2.
\eeq
 As a consequence, the two first terms in \eqref{err20} follows from \eqref{core0}, \eqref{spliterr} and \eqref{prjpi}.

For the last term in \eqref{err20}, we rewrite \eqref{midres0} as  follows 
\beq\label{sq10}
 \intd{K(|\nabla \rho|)\nabla \rho -K(|\nabla \rho_h|)\nabla \rho_h ,\nabla \rho -\nabla\rho_h }
  =(\phi(\rho_t -\rho_{h,t}), \vartheta_h) + \intd{K(|\nabla \rho|)\nabla \rho -K(|\nabla \rho_h|)\nabla \rho_h ,\nabla \theta}.
\eeq
In virtue of the triangle inequality, and H\"older's inequality, we have      
\beq\label{ineq0}
(\rho_t -\rho_{h,t}, \vartheta_h) \le  (|\rho_t|+|\rho_{h,t}|,|\vartheta_h|) \le (\norm{\rho_t}+\norm{\rho_{h,t}} )\norm{\vartheta_h}.
\eeq

It follows from \eqref{sq10}, \eqref{Kl}, \eqref{Kr},  and \eqref{ineq0} that
\begin{align*}
\int_\Omega \Big(\int_0^1 K(|\gamma(s)|) ds\Big) |\nabla \rho-\nabla \rho_h|^2 dx \le C(\norm{\rho_t}+\norm{\rho_{h,t}} )\norm{\vartheta_h} + C \norm{\nabla \theta}^2  .
\end{align*}
Thanks to \eqref{Est4rhot} and \eqref{Est-rhot}, there is a positive constant $C$ independence of $h$ such that 
\beq
 \int_\Omega \Big(\int_0^1 K(|\gamma(s)|) ds\Big)|\nabla \rho-\nabla \rho_h|^2 dx \le C  \norm{\vartheta_h} +C \norm{\nabla \theta}^2.
\eeq
This and \eqref{mL2} show that
\beq\label{err4m}
\norm{\zeta_h}^2 \le C\int_\Omega \Big(\int_0^1 K(|\gamma(s)|) ds\Big) |\nabla \rho-\nabla \rho_h |^2 dx \le C  \norm{\vartheta_h} +C \norm{\nabla \theta}^2.
\eeq

Due to \eqref{core0}, \eqref{err4m} and the fact that $\norm{\nabla \theta}\le Ch^{k}\norm{\rho}_{k+1}$, we obtain \eqref{err20}.
\end{proof}

\begin{theorem}\label{par-thm} Given $0<t_*<T$. Let $(\rho_{h,i}, m_{h,i}), i=1,2$ be two solutions to problems \eqref{semidiscreteform} corresponding to vector coefficients  $\vec a_i$ of Forchheimer polynominal $g(s, \vec a_i)$ in \eqref{eq2}. Suppose that each $(\rho_i, m_i)\in L^\infty(I; H^{k+1}(\Omega) ) \times \left(L^2(I; L^2(\Omega))\right)^d$  and $\rho_{t,i}\in L^2(I; H^{k+1}(\Omega) )$. Then, there exists a constant positive constant $C$ independent of $h$ and $|\vec a_1 -\vec a_2|$ such that   
\beq\label{mrhoDif}
\norm{\rho_{h,1} -\rho_{h,2}}_{L^\infty(I;L^2(\Omega))}+ \norm{m_{h,1} - m_{h,2}}_{L^\infty(t_*,T;L^2(\Omega))}^2 \le C ( h^k+ |\vec a_1 -\vec a_2| ).
\eeq
\end{theorem}
\begin{proof}
The triangle inequality shows that 
$$
\norm{\rho_{h,1} -\rho_{h,2}} + \norm{m_{h,1} -m_{h,2}}^2  \le 4 \left( \sum_{i=1,2}(\norm{\rho_{h,i} -\rho_i} + \norm{m_{h,i} -m_i}^2) +\norm{\rho_1 -\rho_2 }+\norm{m_1 -m_2}^2\right).
$$
Then by using \eqref{err20} to treat the sum-term and \eqref{ssc2} to the last terms we obtain \eqref{mrhoDif}.   
\end{proof}
\section{Error analysis for fully discrete method}\label{err-ful-sec}
 In analyzing this method, we proceed in a similar fashion as for the semidiscrete method. We derive an error estimate for the fully discrete time Galerkin approximation of the differential equation. First, we give some uniform stability results  that are crucial in
getting the  convergence results.

\begin{lemma}[Stability]\label{stab-appr} Let  $(\rho_h^i, m_h^i)$ solve the fully discrete finite element
approximation \eqref{fullydiscreteform} for each time step $i=1,2,\ldots ,N$. There exists a positive constant $C$ independent of $t,i,\Delta t$ such that for $\Delta t$ sufficiently small 
\beq\label{pwBound}
\norm{\rho_h^i }^2 +\norm{m_h^i} \le  C (1-\Dt)^{-i} \norm{\rho^0}^2 +C\Dt \sum_{j=1}^i (1-\Dt)^{-i+j-1} \Big(1+\norm{\psi^j}_{L^\infty(\Gamma)}^{\lambda} + \norm{f^j}^2\Big).  
\eeq
\end{lemma}
 \begin{proof}
Selecting $z=2\nabla \rho_h^i,  r=2\rho_h^i$ in \eqref{fullydiscreteform}, we find that  
\beq
\begin{split}
  2\intd{m_h^i, \nabla\rho_h^i}+2\intd{K(|\nabla\rho_h^i|)\nabla\rho_h^i,\nabla\rho_h^i } =0,\\
2\intd{\phi\frac{ \rho_h^i - \rho_h^{i-1}}{\Delta t }, \rho_h^i }-2\intd{m_h^i,\nabla \rho_h^i }=2\intd{f^i,\rho_h^i}-2\intb{\psi^i,\rho_h^i }. 
\end{split}
\eeq
Adding the two above equations, and using the identity  
 \beqs
 2\Big( \phi (\rho_h^i - \rho_h^{i-1}), \rho_h^i\Big) = \norm{\rho_h^i}_{\phi}^2 - \norm{\rho_h^{i-1}}_{\phi}^2 +\norm{\rho_h^i -\rho_h^{i-1}}_{\phi}^2,
 \eeqs 
 we obtain 
\beq\label{t1}
\norm{\rho_h^i}_{\phi}^2 - \norm{\rho_h^{i-1}}_{\phi}^2 +\norm{\rho_h^i -\rho_h^{i-1}}_{\phi}^2 + 2\Delta t\intd{ K(|\nabla \rho_h^i|)\nabla \rho_h^i,\nabla \rho_h^i}= 2\Delta t (f^i, \rho_h^i )-2\Delta t \intb{\psi^i, \rho_h^i }.  
\eeq
It follows from \eqref{iineq2} that
\beq\label{t2}
 2\Delta t\intd{ K(|\nabla \rho_h^i|)\nabla \rho_h^i,\nabla \rho_h^i} \ge 2c_2\Delta t \norm{\nabla \rho_h^i}_{0,\beta}^{\beta}-2c_2\Delta t |\Omega|. 
\eeq
Using \eqref{sec2} and H\"older's inequality to the RHS of \eqref{t1} shows that
\beq\label{t3}
2\Delta t\Big( \intd{f^i, \rho_h^i}- \intb{\psi^i, \rho_h^i}\Big)\\
\le 2\Dt \norm{\rho_h^i}_{\phi}\norm{f^i}_{\phi^{-1}}+2\Delta t \norm{\psi^i}_{L^\infty(\Gamma)}\Big\{ C_1\norm{\rho_h^i} + \varep\norm{\nabla \rho_h^i}_{0,\beta}^{\beta} +C_1\varep^{-\frac 1{\beta-1}}   \Big\}.
\eeq
Combining \eqref{t1}--\eqref{t3}, then selecting $\varep= \frac{c_2}2 ( \norm{\psi^i}_{L^\infty(\Gamma)}+1)^{-1}$ yields  
\beqs
\begin{split}
&\norm{\rho_h^i}_{\phi}^2 - \norm{\rho_h^{i-1}}_{\phi}^2 +\norm{\rho_h^i -\rho_h^{i-1}}_{\phi}^2 +c_2\Delta t \norm{\nabla \rho_h^i}_{0,\beta}^{\beta}\\
& \qquad\le  2\Dt \norm{\rho_h^i}_{\phi}\norm{f^i}_{\phi^{-1}}+C\Dt \norm{\psi^i}_{L^\infty(\Gamma)}\norm{\rho_h^i} +C\Dt \norm{\psi^i}_{L^\infty(\Gamma)}(\norm{\psi^i}_{L^\infty(\Gamma)}+1)^{\frac 1{\beta-1}}+C\Dt\\
&\qquad\le \Dt\norm{\rho_h^i}_{\phi}^2  + C\Delta t\Big(1+\norm{\psi^i}_{L^\infty(\Gamma)}+\norm{\psi^i}_{L^\infty(\Gamma)}^2+\norm{\psi^i}_{L^\infty(\Gamma)}^{\frac{\beta}{\beta-1} } + \norm{f^i}^2\Big).
 \end{split}
\eeqs
We simplify the RHS of the above estimate using the inequality $\norm{\psi^i}_{L^\infty(\Gamma)}, \norm{\psi^i}_{L^\infty(\Gamma)}^2\le C(1+\norm{\psi^i}_{L^\infty(\Gamma)}^\lambda)$ to obtain     
\beqs
\frac{\norm{\rho_h^i}_{\phi}^2 - \norm{\rho_h^{i-1}}_{\phi}^2}{\Delta t} - \norm{\rho_h^i}_{\phi}^2 +c_2 \norm{\nabla \rho_h^i}_{0,\beta}^{\beta} \le   C\Big(1+\norm{\psi^i}_{L^\infty(\Gamma)}^{\lambda} + \norm{f^i}^2\Big).
\eeqs
 According to discrete Gronwall's inequality in Lemma~\ref{DGronwall},  
\beq\label{dscEst1}
 \norm{\rho_h^i }^2 + \norm{\nabla \rho_h^{i}}_{0,\beta}^{\beta}   \le  C (1-\Dt)^{-i} \norm{\rho_h^0}^2 +\Dt \sum_{j=1}^i (1-\Dt)^{-i+j-1} \Big(1+\norm{\psi^j}{L^\infty(\Gamma)}^{\lambda} + \norm{f^j}^2\Big)  .
 \eeq

Now selecting  $z=m_h^i$ in the first equation \eqref{fullydiscreteform} gives 
$
  \norm{m_h^i}^2 +\intd{K(|\nabla\rho_h^{i}|)\nabla\rho_h^{i},m_i} =0,
$
hence 
\beq\label{dscEst3}
\norm{m_h^i} \le C\norm {K^{1/2}(|\nabla\rho_h^{i}|)\nabla\rho_h^{i} }\le \norm{\nabla\rho_h^{i}}_{0,\beta}^{\beta}.
\eeq
Putting \eqref{dscEst1} and \eqref{dscEst3} together with note that $\norm{\rho_h^0}^2 \le \norm{\rho^0}^2$ implies \eqref{pwBound}. The proof is complete.
\end{proof}

As in the semidiscrete case, we use $\eta = m -\pi m,$ $\zeta_h=m_h-\pi m $,   $\theta=\rho-\pi\rho$, $\vartheta_h=\rho_h-\pi \rho$ and $\eta^i, \theta^i$, $\zeta_h^i,\vartheta_h^i$  be evaluating $\eta, \theta$, $\zeta_h,\vartheta_h$  at the discrete time levels.
We also define  
$$\partial \varphi^n = \frac {\varphi^{n} -\varphi^{n-1} }{\Delta t}.$$
\begin{theorem}\label{Err-ful}
Let $(\rho^i,m^i)$ solve problem \eqref{weakform} and $(\rho_h^i, m_h^i)$ solve the fully  discrete finite element approximation \eqref{fullydiscreteform} for each time step $i$, $i=1,\ldots, N$. Suppose that  $(\rho,m)\in L^\infty(I; H^{k+1}(\Omega))\times \left( L^\infty(I; L^2(\Omega))\right)^d $ and $\rho_{tt} \in L^2(I; L^2(\Omega))$. Then, there exists a positive constant $C$ independent of $h$ and $\Delta t$  such that for $\Delta t$ sufficiently small  
\beq\label{fulerrl2}
 \norm {\rho^i -\rho_h^i} +\norm {m^i -m_h^i}^2  \le C\left( h^{k}+ \sqrt{\Dt} \right).
\eeq  
\end{theorem}
\begin{proof}
Evaluating equation \eqref{weakform} at $t=t_i$ gives
\beq\label{fuldis1}
\begin{split}
 \intd{m^i, z}+\intd{K(|\nabla\rho^i|)\nabla\rho^i, z} =0,\\
\intd{\phi\rho_t^i,r}-\intd{m^i,\nabla r }=\intd{f^i,r}-\intb{\psi^i,r} . 
\end{split}
\eeq  
Subtracting \eqref{fullydiscreteform} from \eqref{fuldis1}, we obtain 
\beq\label{ErrSys}
\begin{split}
\intd{m^i -m_h^i, z} + \intd{K(|\nabla\rho^{i}|)\nabla\rho^{i}- K(|\nabla\rho_h^{i}|)\nabla\rho_h^{i}, z} =0,\quad \forall z\in \M_h,\\
\intd{\phi (\rho_t^i - \partial \rho_h^i),r}-\intd{m^i- m_h^i,\nabla r }=0 , \quad \forall r\in \Ro_h. 
\end{split}
\eeq
Choosing $r=-\vartheta_h^i, z=\nabla r$, and adding the two equations shows that   
\beq\label{errEq}
\intd{\phi (\rho_t^i-\partial\rho_h^i),\vartheta_h^i} + \intd{K(|\nabla\rho^{i}|)\nabla\rho^{i}- K(|\nabla\rho_h^{i}|)\nabla\rho_h^{i}, \nabla \vartheta_h^i}  =0 . 
\eeq 
Since  $\rho_t^i-\partial\rho_h^i = \rho_t^i-\partial \rho^i+ \partial \theta^i - \partial \vartheta_h^i,$  
we rewrite \eqref{errEq} in the form  
\beq\label{eereq}
\begin{split}
&(\partial \vartheta_h^i , \vartheta_h^i ) +  \left(K(|\nabla \rho^i|)\nabla \rho^i- K(|\nabla \rho_h^{i}|)\nabla \rho_h^{i}, \nabla \rho^i -\nabla\rho_h^{i} \right)\\
&\qquad\qquad=\left(K(|\nabla \rho^i|)\nabla \rho^i- K(|\nabla \rho_h^{i}|)\nabla \rho_h^{i}, \nabla \theta^i\right)+\intd{\phi (\rho_t^i-\partial \rho^i),\vartheta_h^i }+\intd{\phi\partial \theta^i ,\vartheta_h^i}.  
\end{split} 
\eeq
We will evaluate \eqref{eereq}  term by term. 

\textbullet\quad  For the first term, we use the identity 
\beq\label{rhs1}
\begin{split}
(\partial \vartheta_h^i , \vartheta_h^i )= \frac 1{2\Delta t}\left(\norm{\vartheta_h^i}_{\phi}^2  -\norm{\vartheta_h^{i-1}}_{\phi}^2\right)+\frac{\Delta t}{2}\norm{\partial \vartheta_h^i }_{\phi}^2 . 
\end{split}
\eeq

\textbullet\quad For the second term, the monotonicity of $K(\cdot)$ in \eqref{monotone0} yields 
 \beq\label{rhs2} 
 (K(|\nabla \rho^i|)\nabla \rho^i -K(|\nabla \rho_h^i|)\nabla \rho_h^i, \nabla\rho^i -\nabla\rho_h^i) \ge (\beta-1) \int_\Omega \Big(\int_0^1 K(|\gamma^i(s)|) ds\Big)|\nabla\rho^i -\nabla \rho_h^i|^2dx, 
 \eeq
 where  $ \gamma^i(s)=s\nabla \rho^i +(1-s)\nabla\rho_h^i. $
  
\textbullet\quad For the third term, using \eqref{Kr} gives       
\beq\label{lhs1}
(K(|\nabla \rho^i|)\nabla \rho^i -K(|\nabla \rho_h^i|)\nabla \rho_h^i ,\nabla \theta^i)
\le \frac{\beta-1} 2 \int_\Omega  \Big(\int_0^1 K(|\gamma^i(s)|) ds\Big)|\nabla \rho^i-\nabla \rho_h^i|^2 dx +C \norm{\nabla \theta^i}^2 .
\eeq
 
 \textbullet\quad For the fourth term, using Taylor expand we see that  
  \beq\label{lhs2}
  \begin{split}
\intd{\phi (\rho_t^i-\partial \rho^i),\vartheta_h^i }
 &\le \frac {1}{\Delta t} \norm{\int_{t_{i-1}}^{t_i} \rho_{tt}(\tau) (\tau-t_{i-1})   d\tau}_{\phi}\norm{\vartheta_h^i}_{\phi}\\
 &\le \frac {1}{\Delta t }\left(\int_{t_{i-1}}^{t_i} \norm{\rho_{tt}(\tau)}_{\phi}^2 d\tau\right)^{\frac1 2}\left(\int_{t_{i-1}}^{t_i} (\tau-t_{i-1})^2   d\tau\right)^{\frac 12} \norm{\vartheta_h^i}_{\phi}\\
 &\le C\Delta t \int_{t_{i-1}}^{t_i} \norm{\rho_{tt}(\tau)}_{\phi}^2 d\tau +\frac 1 4\norm{\vartheta_h^i}_{\phi}^2.
 \end{split}
 \eeq
 
 \textbullet\quad For the last term, using Young's and H\"older's inequalities, we find that  
 \beq\label{lhs3}
 \begin{aligned}
 \intd{\phi\partial \theta^i ,\vartheta_h^i}&\le C\norm{\partial \theta^i}_{\phi}^2 +\frac 1 4 \norm{\vartheta_h^i}_{\phi}^2= C\norm{\int_{t_{i-1}
}^{t_i}\theta_t(\tau) d\tau}_{\phi}^2   +\frac 1 4 \norm{\vartheta_h^i}_{\phi}^2\\
&\le C\Big(\int_{t_{i-1}}^{t_i} \norm{\theta_t(\tau)} d\tau\Big)^2   +\frac 1 4\norm{\vartheta_h^i}_{\phi}^2\\
&\le C\Dt \int_{t_{i-1}}^{t_i}\norm{\theta_t(\tau)}^2 d\tau   +\frac 1 4 \norm{\vartheta_h^i}_{\phi}^2.
 \end{aligned}
 \eeq
 
   In view of \eqref{rhs1}--\eqref{lhs3}, \eqref{eereq} yields       
\begin{multline}\label{vith}
 \frac 1{2\Delta t}\left(\norm{\vartheta_h^i}_{\phi}^2  -\norm{\vartheta_h^{i-1}}_{\phi}^2\right)+   \frac{\beta-1} 2 \int_\Omega  \Big(\int_0^1 K(|\gamma^i(s)|) ds\Big)|\nabla \rho^i-\nabla \rho_h^i|^2 dx \\ 
\le C\Dt \int_{t_{i-1}}^{t_i} \norm{\rho_{tt}(\tau)}^2 d\tau +\frac 1 2\norm{\vartheta_h^i}_{\phi}^2+ C\Dt  \int_{t_{i-1}}^{t_i}\norm{\theta_t(\tau)}^2 d\tau +C\norm{\nabla\theta^i}^2,
\end{multline}
which leads to 
\beq
 \frac {\norm{\vartheta_h^i}_{\phi}^2  -\norm{\vartheta_h^{i-1}}_{\phi}^2}{\Dt} -  \norm{\vartheta_h^i}_{\phi}^2  
\le C\Dt \int_{t_{i-1}}^{t_i} \norm{\rho_{tt}(\tau)}^2 d\tau + C\Dt  \int_{t_{i-1}}^{t_i}\norm{\theta_t(\tau)}^2 d\tau +C\norm{\nabla\theta^i}^2 .
\eeq
By mean of discrete Gronwall's inequality in Lemma~\ref{DGronwall} and the fact $\vartheta_h^0=0$, we find that  
\beq\label{comb3}
\begin{split}
\norm{\vartheta_h^i}^2 &\le C\sum_{j=1}^i (1-\Dt)^{-i+j-1} \Big(\Dt \int_{t_{j-1}}^{t_j} \norm{\rho_{tt}(\tau)}^2 d\tau + \Dt  \int_{t_{j-1}}^{t_j}\norm{\theta_t(\tau)}^2 d\tau +\norm{\nabla\theta^j}^2 \Big)\\
&\le C \Big(\Dt \int_0^T \norm{\rho_{tt}(\tau)}^2 d\tau + \Dt  \int_0^T \norm{\theta_t(\tau)}^2 d\tau +\sum_{j=1}^N \norm{\nabla\theta^j}^2 \Big)
\end{split}
\eeq
since $ (1-\Dt)^{-i+j-1} \le  (1-\Dt)^{-N} \le e^{\frac{N\Dt}{1-\Dt}} \le e^{c_*T} $.

The first part of \eqref{fulerrl2} follows from combining \eqref{comb3} and the triangle inequality
$ \norm{\rho^i-\rho_h^i}\le \norm{\theta^i} +\norm{\vartheta_h^i}. $ 
 
 To the other part of \eqref{fulerrl2},we first estimate the term $ \int_\Omega \Big(\int_0^1 K(|\gamma^i (s)|) ds\Big)|\nabla \rho^i-\nabla \rho_h^i|^2 dx$ by rewriting \eqref{errEq} in the form 
\beq\label{ab}
 \intd{K(|\nabla \rho^i|)\nabla \rho^i -K(|\nabla \rho_h^i|)\nabla \rho_h^i,\nabla \rho^i -\nabla\rho_h^i }
  =(\phi(\rho_t^i -\partial\rho_h^i), \vartheta_h^i) + \intd{K(|\nabla \rho^i|)\nabla \rho^i -K(|\nabla \rho_h^i|)\nabla \rho_h^i ,\nabla \theta^i}.
\eeq
In virtue of the triangle inequality, and H\"older's inequality      
\beq\label{abc}
(\phi(\rho_t^i -\partial\rho_h^i), \vartheta_h^i) \le  (\phi(|\rho_t^i|+|\partial\rho_h^i|),|\vartheta_h^i|) \le (\norm{\rho_t^i}_{\phi}+\norm{\partial\rho_h^i}_{\phi} )\norm{\vartheta_h^i}_{\phi}.
\eeq

It follows from \eqref{ab}, \eqref{Kl}, \eqref{Kr} and \eqref{abc} that
\begin{align*}
\int_\Omega \Big(\int_0^1 K(|\gamma^i(s)|) ds\Big) |\nabla \rho^i-\nabla \rho_h^i|^2 dx \le C(\norm{\rho_t^i}_{\phi}+\norm{\partial\rho_h^i}_{\phi} )\norm{\vartheta_h^i}_{\phi} + C \norm{\nabla \theta^i}^2  .
\end{align*}
Note that $\norm{\partial\rho_h^i}_{\phi}\le C\sup_{t\in[t_1,T]}\norm{\rho_{h,t}(t)}$. 
Thanks to \eqref{Est4rhot} and equivalent norm, we find that 
\beq
 \int_\Omega \Big(\int_0^1 K(|\gamma^i (s)|) ds\Big)|\nabla \rho^i-\nabla \rho_h^i|^2 dx \le C  (\norm{\vartheta_h^i}_{\phi} + \norm{\nabla \theta^i}^2)\le C  (\norm{\vartheta_h^i} + \norm{\nabla \theta^i}^2).
\eeq

Then we chose $z=\zeta_h^i$ as the test function in the first equation of \eqref{ErrSys} to obtain 
\beq\label{mErrL2}
\begin{split}
\norm{\zeta_h^i}^2 &= - \intd{K(|\nabla\rho^i|)\nabla\rho^i - K(|\nabla\rho_h^i|)\nabla\rho_h^i, \zeta_h^i} \le C \int_\Omega \Big(\int_0^1 K(|\gamma^i (s)|) ds\Big) |\nabla \rho^i - \nabla \rho_h^i || \zeta_h^i| dx\\  
&\le C  \int_\Omega \Big(\int_0^1 K(|\gamma^i (s)|) ds\Big) |\nabla \rho^i-\nabla \rho_h^i |^2 dx dt\\
&\le C  (\norm{\vartheta_h^i} + \norm{\nabla \theta^i}^2).
\end{split}
\eeq   
The second part of \eqref{fulerrl2} follows from the combination of \eqref{comb3}, \eqref{mErrL2} and the approximation properties. The proof is complete.  
\end{proof}

The following theorem about an error estimate for $(\rho_h^i, m_h^i)$ is obtained by using the same manner as in the proof of Theorem~\ref{par-thm}.     

\begin{theorem}\label{Dep-ful} Let $(\rho_j^i,m_j^i)$, j=1,2 solve problem \eqref{weakform} and $(\rho_{h,j}^i, m_{h,j}^i)$ solve the fully  discrete finite element approximation \eqref{fullydiscreteform} corresponding to vector coefficients  $\vec a_j$ of Forchheimer polynominal $g(s, \vec a_j)$ in \eqref{eq2}  for each time step $i$, $i=1,\ldots, N$. Suppose that each $(\rho_j, m_j)\in L^\infty(I; H^{k+1}(\Omega) ) \times \left(L^\infty(I; L^2(\Omega))\right)^d$  and $\rho_{tt,j}\in L^\infty(I; L^2(\Omega) )$. There exists a positive constant  $C$  independent of $h,$ $\Delta t$ and $ | \vec a_1 -\vec a_2 |$  such that if  the $\Delta t$ is sufficiently small then 
 \beq\label{ErrContEst}
 \norm {\rho_{h,1}^i -\rho_{h,2}^i} + \norm {m_{h,1}^i -m_{h,2}^i}^2 \le C\left( h^{k}+|\vec a_1-\vec a_2|+ \sqrt{\Dt} \right).
\eeq  
 \end{theorem}
\section{Numerical results} \label{Num-result}

In this section we carry out numerical experiments using mixed finite element approximation to solve problem \eqref{fullydiscreteform} in two dimensional region. For simplicity, the region of examples are unit square $\Omega=[0,1]^2$.    
The triangularization in region $\Omega$ is uniform subdivision in each dimension. We use the piecewise-linear elements for the both density and momentum variables. Our problem is solved at each time level starting at $t=0$ until the given final time $T=1$.  A Newton iteration was used to solve the nonlinear equation generated at each time step.

   The numerical examples in this section are constructed in two categories:    
   \begin{itemize}
   \item Examples 1A and 2A are used to study the convergence rates of the method proposed in the paper. We test the convergence of our method with the Forchheimer two-term law $g(s)=1+s$. Equation \eqref{Kdef} $sg(s)=\xi,$ $s\ge0$ gives
 $ s= \frac {-1 +\sqrt{1+4\xi}}{2}$ and hence 
$
K(\xi) =\frac1{g(s(\xi)) } =\frac {2}{1+\sqrt{1+4\xi}}.
$  

   \item Examples 1B and 2B are used to study the dependence of solution on physical parameters. We test the convergence of our method with the Forchheimer two-term law $g(s)=1+0.95s$. In this case  
$
K(\xi) =\frac {10}{5+\sqrt{25+95\xi}}.
$  
 \end{itemize}
   
 {\bf Example 1.} To test the convergence rates, we choose the the analytical solution  
\beqs
\rho(x,t)=e^{-2t}(x_1+ x_2) \quad\text{ and  }\quad m(x,t) =-\frac{2e^{-2t}(1; 1)}{1+\sqrt{1+4\sqrt{2}e^{-2t}} } \quad \forall  (x,t)\in [0,1]^2\times (0,1].
\eeqs
For simplicity, we take $\phi(x)=1$ on $\Omega$. The forcing term $f$ is determined from equation $\rho_t +\nabla \cdot m = f$. Explicitly, 
$$
f(x,t)= -2e^{-2t}(x_1+ x_2).
$$
The initial condition and boundary condition are determined according to the analytical solution as follows: 
\begin{align*}
 \rho^0(x) =x_1+x_2, \quad \psi(x,t)=\begin{cases}  
\frac{2e^{-2t}}{1+\sqrt{ 1+4\sqrt{2}e^{-2t}}}  & \text { on } \{0\}\times[0,1] \text { or } [0,1]\times\{0\},\\
-\frac{2e^{-2t}}{1+\sqrt{ 1+4\sqrt{2}e^{-2t}}}  & \text { on } \{1\}\times[0,1] \text{ or } [0,1]\times\{1\}.
 \end{cases}
\end{align*}
The numerical results are listed in Table 1A. 

 \vspace{0.2cm}  
\begin{center}
\begin{tabular}{l| c| c| c| c}
\hline
N    & \qquad $\norm{\rho-\rho_h}$ \qquad  &  \quad Rates  \quad&  \quad$\norm{m-m_h}$ \quad & \quad  Rates \quad  \\
\hline
4	&$1.865E-01$        	&--          		& $7.557E-01$		&--\\
8	&$1.529E-01$	   		&$0.287$  		& $7.776E-01$  	&$0.157$\\
16	&$9.651E-02$	   		&$0.664$    	& $5.784E-01$ 	&$0.228$\\
32	&$5.724E-02$			&$0.754$  		& $4.910E-01$ 	&$0.236$\\
64	&$3.389E-02$	   		&$0.756$		& $4.140E-01$		&$0.246$\\
128	&$2.000E-02$			&$0.761$		& $3.460E-01$		&$0.259$\\
256	&$1.179E-02$  		&$0.762$		& $2.981E-01$		&$0.259$\\
512	&$6.901E-03$  		&$0.773$		& $2.407E-01$		&$0.264$\\
\hline
\end{tabular}
\vspace{0.2cm} 

Table 1A. {\it Convergence study for Darcy-- Forchheimer flows using mixed FEM in 2D.}
\end{center} 

Next, we consider a small change in coefficients of Forchheimer polynomial $g$, namely  $g(s)=1+0.95s$. In this case,  the analytical solution is chosen by 
\beqs
\rho(x,t)=e^{-2t}(x_1+x_2) \quad\text{ and  }\quad m(x,t) =-\frac{10e^{-2t}(1; 1)}{5+\sqrt{25+95\sqrt{2}e^{-2t}} } \quad \forall  (x,t)\in [0,1]^2\times (0,1].
\eeqs
The forcing term $f$, initial condition and boundary condition accordingly are  
\begin{align*}
&f(x,t)= -2e^{-2t}(x_1+x_2),\quad \rho^0(x) =x_1+x_2, \quad \psi(x,t)=\begin{cases}  
\frac{10e^{-2t}}{5+\sqrt{ 25+95\sqrt{2}e^{-2t}}}  & \text { on } \{0\}\times[0,1] \text { or } [0,1]\times\{0\},\\
-\frac{2e^{-2t}}{5+\sqrt{ 25+95\sqrt{2}e^{-2t}}}  & \text { on } \{1\}\times[0,1] \text{ or } [0,1]\times\{1\}.
\end{cases}
\end{align*}
We use $\norm{\rho_{1,h}-\rho_{2,h}}$ and $\norm{m_{1,h}-m_{2,h}}$ as the criterion to measure the dependence of solutions on the coefficients of $g$. The numerical results are listed in Table 1B.

\vspace{0.3cm}  
\begin{center}
\begin{tabular}{l| c| c| c| c}
\hline
N    & \qquad $\norm{\rho_{1,h}-\rho_{2,h}}$ \qquad  &  \quad Rates  \quad&  \quad$\norm{m_{1,h}-m_{2,h}}$ \quad & \quad  Rates \quad  \\
\hline
4		&$7.610E-4$      	&--   	        		& $1.640E-3$			&--\\
8		&$5.666E-4$	   	&$0.426$  			& $1.427E-3$  			&$0.131$\\
16		&$4.207E-4$		&$0.430$    		& $1.236E-3$ 			&$0.207$\\
32		&$2.705E-4$		&$0.637$   		& $1.050E-3$ 			&$0.235$\\
64		&$1.634E-4$		&$0.727$			& $8.890E-4$			&$0.240$\\
128		&$9.736E-5$       	&$0.747$			& $7.509E-4$ 			&$0.244$\\
256  	&$5.792E-5$       	&$0.749$			& $6.316E-4$			&$0.250$\\
512  	&$3.445E-5$       	&$0.750$			& $5.312E-4$			&$0.250$\\
\hline
\end{tabular}

\vspace{0.2cm}
Table 1B. {\it Study the dependence of solution of Darcy--Forchheimer flows using mixed FEM in 2D.}
\end{center}
   
   {\bf Example 2.} In this example, we still take $\phi(x)=1$ on $\Omega$. The analytical solution is 
\beqs
\rho(x,t)=e^{-t}w^2(x) \quad\text{ and  }\quad m(x,t) =-\frac{4e^{-t}(x_1; x_2)}{1+\sqrt{1+8e^{-t}w(x)} } \quad \forall  (x,t)\in [0,1]^2\times (0,1],
\eeqs
where $w(x)=\sqrt{x_1^2+x_2^2}$.  The forcing term $f$, initial condition $\rho^0(x)$ and boundary condition $\psi(x,t)$ are as follows 
\begin{align*}
&f(x,t)= -e^{-t}w^2(x)+\frac{16e^{-2t}w^2(x)}{w(x) \sqrt{1+8e^{-t}w(x)}\left(1+\sqrt{1+8e^{-t}w(x)}\right)^2 }-\frac{8e^{-t}}{1+\sqrt{1+8e^{-t}w(x)}},\\
&\hspace{3cm}\rho^0(x) =w^2(x), \quad \psi(x,t)=\begin{cases}  
0  & \text { on } \{0\}\times[0,1] \text{ and }  [0,1]\times\{0\},\\
\frac{-4e^{-t}}{ 1+\sqrt{1+8e^{-t}\sqrt{1+x_2^2}}  }  & \text { on } \{1\}\times[0,1],\\
 \frac{-4e^{-t}}{ 1+\sqrt{1+8e^{-t}\sqrt{x_1^2+1}} }  & \text { on } [0,1]\times\{1\}.
\end{cases}
\end{align*}
The numerical results are listed in Table 2A. 

 \vspace{0.2cm}  
\begin{center}
\begin{tabular}{l| c| c| c| c}
\hline
N    & \qquad $\norm{\rho-\rho_h}$ \qquad  &  \quad Rates  \quad&  \quad$\norm{m-m_h}$ \quad & \quad  Rates \quad  \\
\hline
4	&$2.331e-01$        	&--          		& $9.742E-01$		&--\\
8	&$1.651E-01$	   		&$0.489$  		& $8.926E-01$  	&$0.126$\\
16	&$1.027E-01$	   		&$0.684$    	& $7.511E-01$ 		&$0.249$\\
32	&$6.249E-02$			&$0.717$  		& $6.320E-01$ 	&$0.249$\\
64	&$3.756E-02$	   		&$0.734$		& $5.318E-01$		&$0.249$\\
128	&$2.245E-02$			&$0.742$		& $4.475E-01$		&$0.249$\\
256	&$1.338E-02$  		&$0.746$		& $3.765E-01$		&$0.249$\\
512	&$7.968E-03$  			&$0.748$		& $3.168E-01$		&$0.249$\\
\hline
\end{tabular}

\vspace{0.2cm}
Table 2A. {\it Convergence study for Darcy-- Forchheimer flows using mixed FEM in 2D.}
\end{center} 

\vspace{0.2cm}

For $g(s)=1+0.95s$, the analytical solution is chosen by 
\beqs
\rho(x,t)=e^{-t}w^2(x) \quad\text{ and  }\quad m(x,t) =-\frac{20e^{-t}(x_1; x_2)}{5+\sqrt{25+190e^{-t}w(x)} } \quad \forall  (x,t)\in [0,1]^2\times (0,1].
\eeqs
The forcing term $f,$ initial condition and boundary condition accordingly are  
\begin{align*}
&f(x,t)= -e^{-t}w^2(x)+\frac{1900 e^{-2t}w^2(x)}{w(x) \sqrt{25+190e^{-t}w(x)}\left(5+\sqrt{25+190e^{-t}w(x)}\right)^2 }-\frac{40e^{-t}}{5+\sqrt{25+190e^{-t}w(x)}},\\
&\hspace{3cm}\rho^0(x) =w^2(x), \quad \psi(x,t)=\begin{cases}  
0  & \text { on } \{0\}\times[0,1] \text{ and }  [0,1]\times\{0\},\\
\frac{-20e^{-t}}{ 5+\sqrt{25+190e^{-t}\sqrt{1+x_2^2}}  }  & \text { on } \{1\}\times[0,1],\\
 \frac{-20e^{-t}}{ 5+\sqrt{25+190e^{-t}\sqrt{x_1^2+1}} }  & \text { on } [0,1]\times\{1\}.
\end{cases}
\end{align*}

The numerical results are listed in Table 2B.

\vspace{0.2cm}
\begin{center}
\begin{tabular}{l| c| c| c| c}
\hline
N    & \qquad $\norm{\rho_{1,h}-\rho_{2,h}}$ \qquad  &  \quad Rates  \quad&  \quad$\norm{m_{1,h}-m_{2,h}}$ \quad & \quad  Rates \quad  \\
\hline
4		&$1.950E-03$      	&--   	        		& $4.430E-03$			&--\\
8		&$1.325E-03$	   	&$0.557$  			& $4.008E-03$  		&$0.115$\\
16		&$8.858E-04$		&$0.581$    		& $3.501E-03$ 		&$0.195$\\
32		&$5.712E-04$		&$0.633$   		& $2.975E-03$ 		&$0.235$\\
64		&$3.580E-04$		&$0.674$			& $2.507E-03$			&$0.247$\\
128		&$2.204E-04$       	&$0.700$			& $2.108E-03$ 		&$0.250$\\
256  	&$1.341E-04$       	&$0.717$			& $1.772E-03$			&$0.251$\\
512  	&$8.096E-05$       	&$0.728$			& $1.489E-03$			&$0.251$\\
\hline
\end{tabular}

\vspace{0.2cm}

Table 2B. {\it Study the dependence of solution of Darcy--Forchheimer flows using mixed FEM in 2D.}
\end{center}
%

\bibliographystyle{siam}
\def\cprime{$'$} \def\cprime{$'$} \def\cprime{$'$}

\end{document}